\documentclass{amsart}

\newtheorem{lem}{Lemma}[section] 
\newtheorem{prp}[lem]{Proposition}
\newtheorem{thm}[lem]{Theorem} 
\newtheorem{crl}[lem]{Corollary} 
\newtheorem{defi}[lem]{Definition} 
\newtheorem{nota}[lem]{Notation}

\newtheorem{rque}[lem]{Remark}

\newtheorem{exe}[lem]{Example}

\numberwithin{equation}{section}

\newcommand\N{\mathbb{N}}
\newcommand\Z{\mathbb{Z}}

\newcommand\Pol{\text{Pol}}

\DeclareMathOperator{\Irr}{Irr}
\DeclareMathOperator{\Tr}{Tr}
\DeclareMathOperator{\id}{id}
\DeclareMathOperator{\Rep}{Rep}
\DeclareMathOperator{\Hom}{Hom}

\usepackage{amsmath,amssymb,amsthm}
\usepackage{mathrsfs}
\usepackage{graphicx}
\usepackage{xcolor}
\usepackage{amsmath}
\usepackage[all]{xy}

\title[FREE WREATH PRODUCT BY QUANTUM PERMUTATION GROUPS]{\textbf{Free wreath product
quantum groups: the monoidal category, approximation properties and free probability}}
\author{Fran\c cois Lemeux and Pierre Tarrago}
\curraddr{}
\thanks{}

\begin{document}
\maketitle

\begin{abstract}
In this paper, we find the fusion rules for the free wreath product quantum groups $\mathbb{G}\wr_*S_N^+$ for all compact matrix quantum groups of Kac type $\mathbb{G}$ and $N\ge4$. This is based on a combinatorial description of the intertwiner spaces between certain generating representations of $\mathbb{G}\wr_*S_N^+$. The combinatorial properties of the intertwiner spaces in $\mathbb{G}\wr_*S_N^+$ then allows us to obtain several probabilistic applications. We then prove the monoidal equivalence between $\mathbb{G}\wr_*S_N^+$ and a compact quantum group whose dual is a discrete quantum subgroup of the free product $\widehat{\mathbb{G}}*\widehat{SU_q(2)}$, for some $0<q\le1$. We obtain as a corollary certain stability results for the operator algebras associated with the free wreath products of quantum groups such as Haagerup property, weak amenability and exactness.
\end{abstract}

\section*{Introduction}
The concept of compact pseudogroup was stated in 1987 by Woronowicz in \cite{woronowicz1987compact}, in an attempt to transform the abstract notion of group structure on noncommutative spaces in a more tractable theory. In particular, it encompasses in a unique framework the generalized Pontrjagin duality introduced by Kac in \cite{VK74} and the non-trivial deformations of a compact Lie group as constructed by Drinfeld and Jimbo (see e.g (\cite{DJ86}). This formalisation, later called compact quantum group, allowed  to find new concrete examples of these noncommutative structures. Woronowicz defined the notion of compact quantum group $\mathbb{G}$ as a $C^{*}$-algebra $C(\mathbb{G})$ endowed with a $*$-homomorphism $\Delta:C(\mathbb{G})\rightarrow C(\mathbb{G})\otimes_{\min} C(\mathbb{G})$ and some additional properties. The main point is that if we consider the commutative image of $C(\mathbb{G})$, we get through the Gelfand theorem the algebra of continuous functions on a classical group.

Besides the original example $SU_{q}(2)$ studied by Woronowicz, many new compact quantum groups were introduced by Wang in \cite{Wang}: namely the compact quantum groups $O_{N}^{+},U_{N}^{+}$ and $S_{N}^{+}$ were defined as ``free counterparts" of the known matrix groups $O_{N},U_{N}$ and $S_{N}$. These examples were a first step for the construction of many new compact quantum groups, either by algebraic constructions (free product introduced by Wang in \cite{Wang2}, free wreath product introduced by Bichon in \cite{Bic04}) or by generalization of the combinatoric involved in the description of certain already known compact quantum group (in particular the easy quantum groups of Banica and Speicher in \cite{BS09}).

In \cite{Wor88} Woronowicz described the Tannaka-Krein duality for a compact quantum group, a fundamental tool to manipulate more effectively these objects and to study their properties: this duality associates to a compact quantum group, a certain $C^{*}-$tensor category coming from its corepresentations. Banica and latter Banica and Speicher discovered that these tensor categories have, in some cases, nice combinatoric descriptions. These descriptions are a very efficient way to recover many algebraic properties of a compact quantum group. For instance, Banica used it in \cite{Ban96}, \cite{MR1484551} and \cite{Ban99} to characterize the irreducible representations and the fusion rules for the quantum groups described in Wang thesis.

One of the main consequences of the definition of a compact quantum group is the existence of a Haar state of the $C^{*}-$algebra $C(\mathbb{G})$ of the quantum group $\mathbb{G}$. This allows to consider $C(\mathbb{G})$ as a noncommutative probability space, and in this setting many results were obtained on the stochastic behavior of some elements of $C(\mathbb{G})$. Once again the main strategy to get some results on the stochastic level is to restrict to the associated $C^{*}$-tensor category through the Tannak-Krein duality. See Banica and Collins \cite{banica2005integration},\cite{banica2007integration}, Banica and Speicher in \cite{BS09}, Banica, Curran and Speicher in \cite{banica2011stochastic}, Kostler and Speicher in \cite{kostler2009noncommutative} for more information on the subject.

Another interesting field of investigation for compact quantum group is the study of the operator algebraic properties of the underlying algebras. One can indeed associate to a compact quantum group $\mathbb{G}$, a universal $C^*$-algebra $C_u(\mathbb{G})$ and a reduced $C^*$-algebra $C_r(\mathbb{G})$, and also a von Neumann algebra $L^{\infty}(\mathbb{G})$. One can wonder which properties are satisfied by these algebras. Banica started these studies in \cite{MR1484551} by proving the simplicity of $C_r(U_N^+)$ by adapting certain argument by Powers for classical free groups. Vergnioux proved in \cite{Ver05O} the property of Akemann-Ostrand for $L^{\infty}(U_N^+)$ and $L^{\infty}(O_N^+)$ and together with Vaes proved the factoriality, fullness and exactess for $L^{\infty}(O_N^+)$ in \cite{VV07}. More recently, in \cite{Bra11} and \cite{Bra12}, Brannan proved the Haagerup property for $L^{\infty}(O_N^+)$, $L^{\infty}(U_N^+)$ and $L^{\infty}(S_N^+)$. Freslon proved the weak-amenability of $L^{\infty}(O_N^+)$, $L^{\infty}(U_N^+)$ in \cite{Fre13} and together with De Commer and Yamashita proved the weak amenability for $L^{\infty}(S_N^+)$ in \cite{CFY13} via a monoidal equivalence argument and by the study of $L^{\infty}(SU_q(2))$. In each of these results, the knowledge of the fusion rules of the compact quantum groups is a crucial tool to prove the properties of the associated reduced $C^*$-algebras and von Neumann algebras.

In this article, we will mainly consider the case of the free wreath product quantum groups defined by Bichon in \cite{Bic04}. It was introduced as the most natural construction to study the free symmetry group of several copies of a same graph. The free wreath product $\wr_{*}$ associates to a compact quantum group $\mathbb{G}$ and a compact subgroup $\mathbb{F}$ of $S_{N}^{+}$ a new compact quantum group $\mathbb{G}\wr_{*}\mathbb{F}$. It is constructed as an analogue of the wreath products of classical groups. An example of this construction was studied by Banica and Vergnioux in \cite{BV09}, and then by Banica, Belinschi, Capitaine and Collins in \cite{BBCC11}: they focused on the free wreath product of the dual of the cyclic group $\mathbb{Z}/s\mathbb{Z}$ with $S_{N}^{+}$. Banica and Vergnioux obtained the fusion rules and Banica, Belinschi, Capitaine and Collins obtained interesting probability results involving free compound poisson variables.

Then the first author generalized these results in \cite{Lem13b} to the case of a free wreath product between the dual $\hat{\Gamma}$ of a discrete groupe $\Gamma$ and $S_{N}^{+}$. Once again he was able to find the fusion rules of the quantum group as well as some operator algebraic properties by using certain results of Brannan on $S_{N}^{+}$ (see \cite{Bra11}).

In this article, we tackle the general problem of the free wreath product between any compact quantum  group of Kac type $\mathbb{G}$ and $S_{N}^{+}$. In particular we give positive answers to the following questions:
\begin{itemize}
\item If one knows the intertwiners space of $\mathbb{G}$, can one describe the intertwiners spaces of $\mathbb{G}\wr_{*}S_{N}^{+}$ ? (See Theorem \ref{bigthm}).
\item Is the fusion semi-ring of $\mathbb{G}\wr_{*}S_{N}^{+}$ a free fusion semi-ring in the sense of \cite{BV09} ? (See Theorem \ref{nonalter}).
\item If the dual of $\mathbb{G}$ has the Haagerup property (resp. is exact, resp. is weakly amenable, resp. has the ACPAP property), does $\mathbb{G}\wr_{*}S_{N}^{+}$ possess the Haagerup property (resp. is exact, resp. is weakly amenable, resp. has the ACPAP property) ? (See Section \ref{secalgop} where we answer these questions and where we also study the converse implications).
\item Is it possible to express the haar state of $\mathbb{G}\wr_{*}S_{N}^{+}$ starting form the haar state of $S_{N}^{+}$ ? (See Subsection \ref{Weingarten}).
\end{itemize}
As a consequence of certain of these results, we also answer to:
\begin{itemize}
\item For a compact quantum subgroup $\mathbb{G}$ of $S_{N}^{+}$ with a fundamental corepresentation $r$, let us denote $\chi_{r}$ the character of this corepresentation. The following question was raised by Banica and Bichon in \cite{banica2007free}:

\textit{On which conditions on $(\mathbb{A},u)$ and $(\mathbb{B},v)$, two quantum subgroups of $S_{N}^{+}$, do we have the equality in law
$$\chi_{u\wr_{*}v}\sim \chi_{u}\boxtimes\chi_{v}$$
with $\boxtimes$ denoting the free multiplicative convolution between two noncommutating variables} ? 

We were able to show that the answer to this question is positive in the case where $\mathbb{B}=S_{N}^{+}$ (See Subsection \ref{char}).
\item $SU_{q}(2)$ is the first example of compact quantum group studied by Woronowicz in \cite{woronowicz1987compact}. The description of the interwiners spaces of $\mathbb{G}\wr_{*}S_{N}^{+}$ yields the folloing result:

\textit{$\mathbb{G}\wr_{*}S_{N}^{+}$ is monoidally equivalent to $\mathbb{H}$, where $\hat{\mathbb{H}}$ a quantum subgroup of $\widehat{\mathbb{G}}\ast\widehat{SU_{q}(2)}$}, (see Theorem \ref{resmono}).

Moreover we provide an explicit description of $\mathbb{H}$. This result implies all the operator algebraic results of Section \ref{secalgop}.
\end{itemize}
Using the description of the intertwiner spaces, Jonas Wahl also obtained in \cite{wahl} further interesting results on the reduced algebra and the von Neumann algebra of $\mathbb{G}\wr S_{N}^{+}$. In particular he obtained the simplicity of the reduced algebra and the unicity of the trace (which he identified with the Haar state), and proved that the von Neumann algebra is a $II_{1}-$factor without property $\Gamma$.\\
The paper is organised as follows : the first section is dedicated to some preliminaries and notations. The second section gives some classical results and proofs that provide an insight on the final general description for the intertwiners spaces of the free wreath products. The description of the intertwiners spaces for the free wreath products $\mathbb{G}\wr_*S_N^+$ is the main result of the third section. In the fourth section, we give the probabilistic applications of that one can deduce from this description. Then in the fifth section, we prove the monoidal equivalence result we mentioned above. This result is fundamental to obtain the operator algebraic consequences of the sixth section which we combined with the results in \cite{CFY13}. Finally in an appendix we provide a dimension formula for the corepresentations of $\mathbb{G}\wr_{*}S_{N}^{+}$.

\section{Preliminaries}
\subsection{Non-crossing partitions, diagrams. Tannaka-Krein duality}
In the following paragraph, we recall a few notions on non-crossing partitions, see e.g. \cite{BV09} for more informations. We also recall some facts on the categorical framework associated to compact quantum groups and non-crossing partitions.
\begin{defi}
We denote by $\mathcal{P}(k,l)$ (\/resp. $NC(k,l)$\/) the set of partitions (resp. non-crossing partitions) between $k$ upper points and $l$ lower points, that is the partitions (resp. non-crossing partitions) of the set $\{1,\dots,k+l\}$ with $k+l$ ordered elements from left to right on top and bottom, with the following pictorial representation:
\smallskip 
\[ \left\{ \begin{array}{ccccc}
.\ & . & . & .\ & .\\
& &  \mathscr{P} & & \\
. & . & . & . &   \\
\end{array} \right\} \]
with $k$-upper points, $l$-lower points and $\mathscr{P}$ is a a diagram composed of strings which connect certain upper and/or lower points (resp. which do not cross one another).
\end{defi}

\begin{nota} We will keep the following notation in the sequel:
\begin{enumerate}
\item[$\bullet$] We will denote by $1_k$ the one-block partition with only lower points which are all connected.
\item[$\bullet$] We will write $p\le q$ if $p$ is a refinement of $q$.
\end{enumerate}

\end{nota}

From now on, we only consider non-crossing partitions, even if certain of the following definitions and propositions hold for any (possibly crossing) partition. 

Non-crossing partitions give rise to new ones by tensor product, composition and involution:

\begin{defi}\label{NC} Let $p\in NC(k,l)$ and $q\in NC(l,m)$. Then, the tensor product, composition and involution of the partitions p,q are obtained by horizontal concatenation, vertical concatenation and symmetry with respect to upside-down turning:
$$p\otimes q=\left\{\mathscr{P}\mathscr{Q}\right\},\ 
pq=\setlength{\unitlength}{0.5cm}
\left\{\ \begin{picture}(1,1)\thicklines
\put(0,0.5){$\mathscr{Q}$}
\put(0,-0.6){$\mathscr{P}$}
\end{picture}\ \right\}-\{\emph{closed blocks, middle points}\},\ 
p^{*}=\{\mathscr{P^{\downarrow}}\}.$$
The operation $p\circ q=\setlength{\unitlength}{0.5cm}
\left\{\ \begin{picture}(1,1)\thicklines
\put(0,0.5){$\mathscr{Q}$}
\put(0,-0.6){$\mathscr{P}$}
\end{picture}\ \right\}$ is only defined if the number of lower points of $q$ is equal to the number of upper points of $p$. When one performs a composition $p\circ q$, and that one identifies the lower points of $p$ with the upper points of $q$, closed blocks might appear, that is strings which are connected neither to the new upper points nor to the new lower points. These blocks are discarded from the final pictorial representation denoted $pq$.
\end{defi}

\begin{exe}
Following the rules stated above (discarding closed blocks and following the lines when one identifies the upper points of $p$ with the lower ones of $q$), we get 
\setlength{\unitlength}{0.5cm}
$$\emph{if}\ \ \ 
p=\left\{\ \begin{picture}(3,3)\thicklines
\put(0,1.2){\line(0,1){0.75}}
\put(2,1.2){\line(0,1){0.75}}
\put(1,1.2){\line(0,1){0.75}}
\put(0,1.228){\line(1,0){1}}
\put(2,1.228){\line(1,0){1}}
\put(3,1.2){\line(0,1){0.75}}

\put(0,-0.878){\line(1,0){1}}
\put(0,-1.6){\line(0,1){0.75}}
\put(1,-1.6){\line(0,1){0.75}}
\put(2,-1.6){\line(0,1){0.75}}

\put(-0.2,2.2){1}
\put(0.8,2.2){2}
\put(1.8,2.2){3}
\put(2.8,2.2){4}
\put(-0.2,-2.2){1}
\put(0.8,-2.2){2}
\put(1.8,-2.2){3}
\end{picture}\ \right\} \ \ \ \emph{ and }\ \ \
q=\left\{\ \begin{picture}(4,3)\thicklines
\put(0,-1.3){\line(0,1){3.2}}
\put(2,1.2){\line(0,1){0.75}}
\put(4,1.2){\line(0,1){0.75}}
\put(1,-1.3){\line(0,1){3.2}}
\put(2,1.228){\line(1,0){1}}
\put(3,1.2){\line(0,1){0.75}}

\put(2,-1.6){\line(0,1){0.75}}
\put(3,-1.6){\line(0,1){0.75}}
\put(2,-0.878){\line(1,0){1}}

\put(-0.2,2.2){1}
\put(0.8,2.2){2}
\put(1.8,2.2){3}
\put(2.8,2.2){4}
\put(3.8,2.2){5}
\put(-0.2,-2.2){1}
\put(0.8,-2.2){2}
\put(1.8,-2.2){3}
\put(2.8,-2.2){4}
\end{picture}\ \right\}\ \ \ \emph{ then }\ \ \
pq=\left\{\ \begin{picture}(4,3)\thicklines
\put(0,1.2){\line(0,1){0.75}}
\put(2,1.2){\line(0,1){0.75}}
\put(1,1.2){\line(0,1){0.75}}
\put(0,1.228){\line(1,0){1}}
\put(2,1.228){\line(1,0){1}}
\put(3,1.2){\line(0,1){0.75}}
\put(4,1.2){\line(0,1){0.75}}

\put(0,-0.878){\line(1,0){1}}
\put(0,-1.6){\line(0,1){0.75}}
\put(1,-1.6){\line(0,1){0.75}}
\put(2,-1.6){\line(0,1){0.75}}

\put(-0.2,2.2){1}
\put(0.8,2.2){2}
\put(1.8,2.2){3}
\put(2.8,2.2){4}
\put(3.8,2.2){5}
\put(-0.2,-2.2){1}
\put(0.8,-2.2){2}
\put(1.8,-2.2){3}
\end{picture}\ \right\}.
$$
\end{exe}
\bigskip\noindent

From non-crossing partitions $p\in NC(k,l)$ naturally arise linear maps $T_p: \mathbb{C}^{N^{\otimes k}}\to \mathbb{C}^{N^{\otimes l}}$:
\begin{defi}\label{canocons}
Consider $(e_i)$ the canonical basis of $\mathbb{C}^N$. Associated to any non-crossing partition $p\in NC(k,l)$ is the linear map $T_p\in B\left(\mathbb{C}^{N^{\otimes k}},\mathbb{C}^{N^{\otimes l}}\right)$:
$$T_p(e_{i_1}\otimes\dots\otimes e_{i_k})=\sum_{j_1,\dots,j_l}\delta_p(i,j)e_{j_1}\otimes\dots\otimes e_{j_l}$$
where $i$ (respectively $j$) is the $k$-tuple $(i_1,\dots,i_k)$ (respectively $l$-tuple $(j_1,\dots,j_l)$) and $\delta_p(i,j)$ is equal to:
\begin{enumerate}
\item $1$ if all the strings of $p$ join equal indices,
\item $0$ otherwise.
\end{enumerate}
\end{defi}

\begin{exe}
We consider an element $p\in NC(4,3)$, choose any tuples $i=(i_1,i_2,i_3,i_4)$ and $j=(j_1,j_2,j_3)$, and put them on the diagram:
\setlength{\unitlength}{0.5cm}
$$
p=\left\{\ \begin{picture}(3.1,3.5)\thicklines
\put(0,0.5){\line(0,1){1.5}}
\put(2,1){\line(0,1){1}}
\put(1,0.5){\line(0,1){1.5}}
\put(0,0.5){\line(1,0){3}}
\put(3,0.5){\line(0,1){1.5}}
\put(0.4,-1.6){\line(0,1){1}}
\put(1.4,-1.6){\line(0,1){2.1}}
\put(2.4,-1.6){\line(0,1){1}}
\put(-0.1,2.2){$\cdot$}
\put(0.9,2.2){$\cdot$}
\put(1.9,2.2){$\cdot$}
\put(2.9,2.2){$\cdot$}
\put(0.4,-2.2){$\cdot$}
\put(1.4,-2.2){$\cdot$}
\put(2.4,-2.2){$\cdot$}
\put(-0.1,3){$i_1$}
\put(0.9,3){$i_2$}
\put(1.9,3){$i_3$}
\put(2.9,3){$i_4$}
\put(0.4,-3){$j_1$}
\put(1.4,-3){$j_2$}
\put(2.4,-3){$j_3$}
\end{picture}\ \right\}\ \text{ Then\ }\ 
\delta_{p}(i,j)=\left\{
    \begin{array}{ll}
        1 & \mbox{\emph{if} } i_1=i_2=i_4=j_2 \\
        0 & \mbox{\emph{otherwise.}}
    \end{array}
\right.
$$
\end{exe}

\begin{exe}\label{idtensor}
We give basic examples of such linear maps
\begin{enumerate}
\item[(i.)] T\_\setlength{\unitlength}{0.5cm}
$
\left\{\ \begin{picture}(0.01,1)\thicklines
\put(0,-0.6){\line(0,1){1.5}}
\end{picture}\ \right\}
=\id_{\mathbb{C}^N}$

\medskip
\item[(ii.)] $T\_{\{\bigcap\}}(1)
=\sum_ae_a\otimes e_a$
\end{enumerate}
\end{exe}

Tensor products, compositions and involutions of diagrams behave as follows with respect to the associated linear maps:

\begin{prp}(\cite[Proposition 1.9]{BS09}\label{tensorcat}
Let $p,q$ be non-crossing partitions and $b(p,q)$ be the number of closed blocks when performing the vertical concatenation (when it is defined). Then:
\begin{enumerate}
\item $T_{p\otimes q}=T_p\otimes T_q$,
\item $T_{p\circ q}=N^{-b(p,q)}T_p\circ T_q$,
\item $T_{p^*}=T_p^*$.
\end{enumerate}
Furthermore, the linear maps $T_p : (\mathbb{C}^N)^{\otimes k}\to(\mathbb{C}^N)^{\otimes l}, p\in NC(k,l)$ are linearly independent if $N\ge4$.
\end{prp}

The Proposition \ref{tensorcat} implies easily that the collection of spaces $\text{span}\{T_p : p\in NC(k,l)\}$ form a $C^*$-tensor category with $\N$ as a set of objects. Furthermore, this tensor category has conjugates since the partitions of type

\setlength{\unitlength}{0.5cm}
$$ 
r=\left\{\ \begin{picture}(6,1.3)\thicklines

\put(0,0.5){\line(1,0){6}}
\put(2.7,1){$\emptyset$}
\put(0,-0.7){\line(0,1){1.22}}
\put(1,-0.7){\line(0,1){0.9}}
\put(5,-0.7){\line(0,1){0.9}}
\put(1,0.17){\line(1,0){4}}
\put(1.4,-0.7){...}
\put(3.6,-0.7){...}
\put(2.4,-0.7){\line(0,1){0.5}}
\put(3.4,-0.7){\line(0,1){0.5}}
\put(2.38,-0.21){\line(1,0){1.05}}
\put(6,-0.7){\line(0,1){1.22}}

\end{picture}\ \right\}\in NC(\emptyset;2k)$$
are non-crossing and since the following conjugate equations hold: 
\begin{equation}\label{eqconj}
(T_r^*\otimes \id)\circ(\id\otimes T_r)=\id=(\id\otimes T_r^*)\circ(T_r\otimes \id).
\end{equation}

We recall that in a $C^*$-tensor category with conjugates, we have the following Frobenius reciprocity theorem (see \cite{Wor88} and \cite{Nesh}) that we will use in the sequel.

\begin{thm}
Let $\mathscr{C}$ be a $C^*$-tensor category with conjugates. If an object $U\in\mathscr{C}$ has a conjugate, with $R$ and $\overline{R}$ solving the conjugate equations (see \cite[Definition 2.2.1]{Nesh}, or (\ref{eqconj}) above), then the map
$$Mor(U\otimes V,W)\to Mor(V,\overline{U}\otimes W), T\mapsto (\id_{\overline{U}}\otimes T)(R\otimes \id_V)$$ is a linear isomorphism with inverse $S\mapsto (\overline{R}^*\otimes \id_W)(\id_U\otimes S)$.
\end{thm}

\subsection{Quantum groups}
In this subsection, we recall a few facts and results about compact quantum groups and about free wreath products by the quantum permutation groups $S_N^+$. 

A compact quantum group is a pair $\mathbb{G}=(C(\mathbb{G}),\Delta)$ where $C(\mathbb{G})$ is a unital (Woronowicz)-$C^*$-algebra and $\Delta : C(\mathbb{G})\to C(\mathbb{G})\otimes_{min} C(\mathbb{G})$ is a unital $*$-homomorphism i.e. they satisfy the coassociativity relation $(\id\otimes\Delta)\circ\Delta=(\Delta\otimes \id)\circ\Delta$, and the cancellation property, that is $span\{\Delta(a)(b\otimes 1) : a,b\in C(\mathbb{G})\}$ and $span\{\Delta(a)(1\otimes b) : a,b\in C(\mathbb{G})\}$ are norm dense in $C(\mathbb{G})\otimes C(\mathbb{G})$. These assumptions allow to prove the existence and uniqueness of a Haar state $h : C(\mathbb{G})\to \mathbb{C}$ satisfying the bi-invariance relations $(h\otimes \id)\circ\Delta(\cdot)=h(\cdot)1=(\id\otimes h)\circ\Delta(\cdot)$. In this paper we will deal with compact quantum groups of Kac type, that is their Haar state $h$ is a trace. 

One can consider the GNS representation $\lambda_h : C(\mathbb{G})\to \mathcal{B}(L^2(\mathbb{G},h))$ associated to the Haar state $h$ of $\mathbb{G}=(C(\mathbb{G}),\Delta)$ and called the left regular representation. 
 The reduced $C^*$-algebra associated to $\mathbb{G}$ is then defined by $C_r(\mathbb{G})=\lambda_h(C(\mathbb{G}))\simeq C(\mathbb{G})/Ker(\lambda_h)$ and the von Neumann algebra by $L^{\infty}(\mathbb{G})=C_r(\mathbb{G})''$. One can prove that $C_r(\mathbb{G})$ is again a Woronowicz-$C^*$-algebra whose Haar state extends to $L^{\infty}(\mathbb{G})$. We will denote simply by $\Delta$ and $h$ the coproduct and Haar state on $C_r(\mathbb{G})$. 

An $N$-dimensional (unitary) corepresentation $u=(u_{ij})_{ij}$ of $\mathbb{G}$ is a (unitary) matrix $u\in M_N(C(\mathbb{G}))\simeq C(\mathbb{G})\otimes \mathcal{B}(\mathbb{C}^N)$ such that for all $i,j\in\{1,\dots,N\}$, one has $$\Delta(u_{ij})=\sum_{k=1}^Nu_{ik}\otimes u_{kj}.$$ The matrix $\overline{u}=(u_{ij}^*)$ is called the conjugate of $u\in M_N(C(\mathbb{G}))$ and in general it is not necessarily unitary even if $u$ is. However all the compact quantum groups we will deal with are of Kac type and in this case the conjugate of a unitary corepresentation is also unitary.

An intertwiner between two corepresentations $$u\in M_{N_u}(C(\mathbb{G})) \text{ and } v\in M_{N_v}(C(\mathbb{G}))$$ is a matrix $T\in M_{N_u,N_v}(\mathbb{C})$ such that $v(1\otimes T)=(1\otimes T)u$.
We say that $u$ is equivalent to $v$, and we note $u\sim v$, if there exists an invertible intertwiner between $u$ and $v$. We denote by $\Hom_{\mathbb{G}}(u,v)$ the space of intertwiners between $u$ and $v$. A corepresentation $u$ is said to be irreducible if $\Hom_{\mathbb{G}}(u,u)=\mathbb{C} \id$. We denote by $\Irr(\mathbb{G})$ the set of equivalence classes of irreducible corepresentations of $\mathbb{G}$.

We recall that $C(\mathbb{G})$ contains a dense $*$-subalgebra denoted by $\Pol(\mathbb{G})$ and linearly generated by the coefficients of the irreducible corepresentations of $\mathbb{G}$. The coefficients of a $\mathbb{G}$-representation $r$ are given by $(\id\otimes\phi)(r)$ for some $\phi\in \mathcal{B}(H_r)^*$ if the corepresentation acts on the Hilbert space $H_r$. This algebra has a Hopf-$*$-algebra structure and in particular there is a $*$-antiautomorphism $\kappa : \Pol(\mathbb{G})\to \Pol(\mathbb{G})$ which acts on the coefficients of an irreducible corepresentation $r=(r_{ij})$ as follows $\kappa(r_{ij})=r_{ij}^*$. This algebra is also dense in $L^2(\mathbb{G},h)$. Since $h$ is faithful on the $*$-algebra $\Pol(\mathbb{G})$, one can identify $\Pol(\mathbb{G})$ with its image in the GNS-representation $\lambda_h(C(\mathbb{G}))$. We will denote by $\chi_r$ the character of the irreducible corepresentation $r\in \Irr(\mathbb{G})$, that is $\chi_r=(\id\otimes \text{Tr})(r)$. 





A fundamental and basic family of examples of compact quantum groups is recalled in the following definition:

\begin{defi}\label{permw}(\cite{Wang})
Let $N\ge2$. $S_N^+$ is the compact quantum group $(C(S_N^+),\Delta)$ where  
$C(S_N^+)$ is the universal $C^*$-algebra generated by $N^2$ elements $u_{ij}$ such that the matrix $u=(u_{ij})$ is unitary and $u_{ij}=u_{ij}^*=u_{ij}^2, \forall i,j$ (i.e. $u$ is a magic unitary) and such that
the coproduct $\Delta$ is given by the usual relations making of $v$ a finite dimensional corepresentation of $C(S_N^+)$, that is $\Delta(u_{ij})=\sum_{k=1}^Nu_{ik}\otimes u_{kj}$, $\forall i,j$.
\end{defi}

In the cases $N=2,3$, one obtains the usual algebras $C(\Z_2), C(S_3)$. 
If $N\ge4$, one can find an infinite dimensional quotient of $C(S_N^+)$ so that $C(S_N^+)$ is not isomorphic to $C(S_N)$, see e.g. \cite{Wang}, \cite{Ban05}. 


In \cite{Wang2}, Wang defined the free product $\mathbb{G}=\mathbb{G}_1*\mathbb{G}_2$ of compact quantum groups, showed that $\mathbb{G}$ is still a compact quantum group and gave a description of the irreducible corepresentations of $\mathbb{G}$ as alternating tensor products of nontrivial irreducible corepresentations. 

\begin{thm}\label{frprodcrucial}(\cite{Wang2}) Let $\mathbb{G}_1$ and $\mathbb{G}_2$ be compact quantum groups. Then the set $\emph{Irr}(\mathbb{G})$ of irreducible corepresentations of the free product of quantum groups $\mathbb{G}=\mathbb{G}_1*\mathbb{G}_2$ can be identified with the set of alternating words in $\emph{Irr}(\mathbb{G}_1)*\emph{Irr}(\mathbb{G}_2)$ and the fusion rules can be recursively described as follows:
\begin{enumerate}
\item[$\bullet$] If the words $x,y\in \emph{Irr}(\mathbb{G})$ end and start in $\emph{Irr}(\mathbb{G}_i)$ and $\emph{Irr}(\mathbb{G}_j)$ respectively with $j\ne i$ then $x\otimes y$ is an irreducible corepresentation of $\mathbb{G}$ corresponding to the concatenation $xy\in \emph{Irr}(\mathbb{G})$.
\item[$\bullet$] If $x=vz$ and $y=z'w$ with $z,z'\in \emph{Irr}(\mathbb{G}_i)$ then $$x\otimes y=\bigoplus_{1\ne t\subset z\otimes z'} vtw\oplus \delta_{\overline{z},z'}(v\otimes w)$$ where the sum runs over all non-trivial irreducible corepresentations $t\in \emph{Irr}(\mathbb{G}_i)$ contained in $z\otimes z'$, with multiplicity.
\end{enumerate}
\end{thm}

In this paper, we are interested in the free wreath product of quantum groups:


\begin{defi}(\cite[Definition 2.2]{Bic04})
Let $A$ be a Woronowicz-$C^*$-algebra, $N\ge2$ and $\nu_i : A\to A^{*N}$ be the canonical inclusion of the $i$-th copy of $A$ in the free product $A^{*N}$, $i=1,\dots,N$. 

The free wreath product of $A$ by $C(S_N^+)$ is the quotient of the $C^*$-algebra $A^{*N}*C(S_N^+)$ by the two-sided ideal generated by the elements $$\nu_k(a)v_{ki}-v_{ki}\nu_k(a), \ \ \ 1\le i,k\le N, \ \ a\in A.$$ It is denoted by $A*_wC(S_N^+)$.
\end{defi}

In the next result, we use the Sweedler notation $\Delta_A(a)=\sum a_{(1)}\otimes a_{(2)}\in A\otimes A$.

\begin{thm}(\cite[Theorem 2.3]{Bic04})
Let $A$ be a Woronowicz-$C^*$-algebra, then free wreath product $A*_wC(S_N^+)$ admits a Woronowicz-$C^*$-algebra structure: if $a\in A$, then
$$\Delta(v_{ij})=\sum_{k=1}^Nv_{ik}\otimes v_{kj}, \forall i,j\in\{1,\dots,N\},$$ $$\Delta(\nu_i(a))=\sum_{k=1}^N\nu_i(a_{(1)})v_{ik}\otimes\nu_k(a_{(2)}),$$
$$\epsilon(v_{ij})=\delta_{ij},\ \epsilon(\nu_i(a))=\epsilon_A(a),\ S(v_{ij})=v_{ji},\ S(\nu_i(a))=\sum_{k=1}^N\nu_k(S_A(a))v_{ki},$$
$$v_{ij}^*=v_{ij},\ \nu_i(a)^*=\nu_i(a^*).$$ Moreover, if $\mathbb{G}$ is a full compact quantum group, then $\mathbb{G}\wr_*S_N^+=(A*_wC(S_N^+),\Delta)$ is also a full compact quantum group.
\end{thm}

\begin{rque}\label{inclinj}
The homomorphisms $\nu_i : A\to A^{*N}\subset A*_{w}C(S_N^+)$ are injective and we have $\nu_i=\pi\circ\bar\nu_i,$ where $$\bar\nu_i=q\circ\nu_i : A\to A*_wC(S_N^+),$$ $q : A^{*N}*C(S_N^+)\to A*_wC(S_N^+)$ is the quotient map and $\pi : A*_wC(S_N^+)=id*\epsilon$. Hence the morphisms $\bar\nu_i : A\to A*_wC(S_N^+)$ are injective.
\end{rque}

Recall that the case of the dual of a discrete group $\mathbb{G}=\widehat{\Gamma}$ is investigated in \cite{Lem13b}. In particular, a description of the irreducible representations is given and several operator algebraic properties are obtained from this description such as factoriality and fulness of the associated von Neumann algebras at least in most cases. In this paper, we shall obtain operator algebraic properties in the more general setting of free wreath product quantum groups $\mathbb{G}\wr_*S_N^+$ with $\mathbb{G}$ compact matrix quantum group of Kac type. To do this, we will use the notion of monoidal equivalence for compact quantum groups.

Two compact quantum groups $\mathbb{G}_1$, $\mathbb{G}_2$ are monoidally equivalent if their representation categories $\text{Rep}(\mathbb{G}_1)$, $\text{Rep}(\mathbb{G}_2)$ are unitarily monoidally equivalent. That is: 

\begin{defi}\label{mono}\cite{BDRV06}
Let $\mathbb{G}_1$, $\mathbb{G}_2$ be two compact quantum groups. We say that $\mathbb{G}_1$ and $\mathbb{G}_2$ are monoidally equivalent, and we write $\mathbb{G}_1\simeq_{mon}\mathbb{G}_2$, if there exists a bijection $\phi : \Irr(\mathbb{G}_1)\to\Irr(\mathbb{G}_2)$ satisfying $\phi(1)=1$, together with linear isomorphisms still denoted $\phi$ 
$$\phi : \Hom_{\mathbb{G}_1}(x_1\otimes\dots\otimes x_k ; y_1\otimes\dots\otimes y_l)\to \Hom_{\mathbb{G}_2}(\phi(x_1)\otimes\dots\otimes \phi(x_k) ; \phi(y_1)\otimes\dots\otimes \phi(y_l))$$
such that:
\begin{enumerate}
\item[$\bullet$] $\phi(\id)=\id$,
\item[$\bullet$] $\phi(S\otimes T)=\phi(S)\otimes\phi(T)$,
\item[$\bullet$] $\phi(S^*)=\phi(S)^*$,
\item[$\bullet$] $\phi(ST)=\phi(S)\phi(T)$,
\end{enumerate}
whenever the formulas make sense.
\end{defi}

We shall prove such a monoidal equivalence for compact quantum groups whose underlying $C^*$-algebras are generated the coefficients of generating corepresentations. It will be enough to construct the maps $\phi$ of Definition \ref{mono} at the level of these generating objects and extend to the completions in the sense of Woronowicz. We refer the reader to \cite{Wor88} and \cite{BDRV06} for notions on monoidal $C^*$-categories associated with compact quantum groups. The reader may also refer to \cite{ML98} for a general introduction to categories. The following result is maybe well known but we include a proof for the convenience of the reader.

\begin{lem}\label{extcomp}
Let $\mathcal{R}_0$ and $\mathcal{S}_0$ be monoidal rigid $C^*$-tensor categories generated by certain objects $\mathcal{R}_0=\langle\alpha : \alpha\in I\rangle$, $\mathcal{S}_0=\langle\beta : \beta\in J\rangle$ and with completions (with respect to direct sums and sub-objects) $\mathcal{R},\mathcal{S}$. Let $\varphi : \mathcal{R}_0\to \mathcal{S}_0$ be an equivalence of such categories. Then there exists a equivalence of categories $\widetilde{\varphi} : \mathcal{R}\to \mathcal{S}$ extending $\varphi$.
\end{lem}
\begin{proof}
We prove that we can extend $\varphi$ to the sub-objects of $\mathcal{R}_0$. We denote the extension $\widetilde{\varphi}$. Let $a,b\in \mathcal{R}$. Then $a$ is contained in a tensor products of generating objects $a\subset \bigotimes\alpha$, and similarly $b\subset\bigotimes\beta$. By definition, there exist isometries $v : a\to\bigotimes\alpha$ and $w : b\to\bigotimes\beta$, $v^*v=\id=w^*w$, such that $vv^*\in \text{End}(\bigotimes\beta)$, $ww^*\in \text{End}(\bigotimes\alpha)$ are the projections onto $a$ and $b$ respectively. 

Now, $\varphi(vv^*)\in \text{End}(\bigotimes\varphi(\alpha))$ is a projection. Hence, there exist an object $\widetilde{\varphi}(a)\in \mathcal{R}$ and an isometry $\widetilde{\varphi}(v) : \widetilde{\varphi}(a)\to \bigotimes\varphi(\alpha)$ such that $\widetilde{\varphi}(v)\widetilde{\varphi}(v)^*=\widetilde{\varphi}(vv^*)$ and $\widetilde{\varphi}(v)\widetilde{\varphi}(v^*)\left(\bigotimes\varphi(\alpha)\right)=\widetilde{\varphi}(a)$. We can proceed similarly for $\varphi(ww^*)\in \text{End}(\bigotimes\varphi(\beta))$ and obtain similar object and isometry $\widetilde{\varphi}(b)$, $\widetilde{\varphi}(w)$. Then, if $S : a\to b$, we can define $\widetilde{\varphi}(S) : \varphi(a)\to\varphi(b)$ by $$\widetilde{\varphi}(S):=\widetilde{\varphi}(w)^*\varphi(wSv^*)\widetilde{\varphi}(v).$$ 

We can make two straightforward remarks. First, a simple calculation shows that $\widetilde{\varphi}(S^*)=\widetilde{\varphi}(S)^*$. Now, notice that if $c=a\oplus b$ with $a,b,c\in \mathcal{R}$ then there exist isometries $u : a\to c$ and $v : b\to c$ such that $uu^*+vv^*=\id$. We have
$$\id=\widetilde{\varphi}(\id)=\widetilde{\varphi}(uu^*+vv^*)=\widetilde{\varphi}(u)\widetilde{\varphi}(u)^*+\widetilde{\varphi}(v)\widetilde{\varphi}(v)^*,$$
so that $\widetilde{\varphi}$ extends $\varphi$ to direct sums.

Let us check that $\widetilde{\varphi}$ is compatible with the composition of morphisms. Compatibility with tensor product is clear and compatibility with involution was mentioned above. If $T : b\to c$ is a morphism between $b,c\in \mathscr{C}$, we can define $\widetilde{\varphi}(T) : \varphi(b)\to\varphi(c)$ in the same way as above starting from an isometry $u : b\to c$,
$$\widetilde{\varphi}(T):=\widetilde{\varphi}(u)^*\varphi(uTw^*)\widetilde{\varphi}(w)$$
and we have $$\widetilde{\varphi}(T\circ S):=\widetilde{\varphi}(u)\varphi(uTS v^*)\widetilde{\varphi}(v) : a\to c.$$
But,
\begin{align*}
\widetilde{\varphi}(T)\circ \widetilde{\varphi}(S)&=\widetilde{\varphi}(u)\varphi(uTw^*)\widetilde{\varphi}(w)\widetilde{\varphi}(w^*)\varphi(wSv^*)\widetilde{\varphi}(v)\\
&=\widetilde{\varphi}(u)\varphi(uTw^*)\widetilde{\varphi}(ww^*)\varphi(wSv^*)\widetilde{\varphi}(v)\\
&=\widetilde{\varphi}(u)\varphi(uTw^*ww^*wSv^*)\widetilde{\varphi}(v)\\
&=\widetilde{\varphi}(u)\varphi(uT Sv^*)\widetilde{\varphi}(v).
\end{align*}
where the third equality above comes from the fact that $\widetilde{\varphi}(ww^*)$ and then also $$\varphi(uTw^*)\widetilde{\varphi}(ww^*)\varphi(wSv^*)$$ are morphisms in the category $\mathcal{R}_0$.
\end{proof}

\section{Classical wreath products by permutation groups.}

In this section we provide a probabilistic formula for the moments of the character coming from certain wreath products of classical groups. This is in particular a hint for the formula in the free case. Recall that we denote by $\mathcal{P}(k)$ the set of all partitions of the set $\{1,\dots,k\}$.

Let $G$ be a classical group, $ n\geq 1$. Then $S_{n}$ acts on $G^{n}$ by the automorphisms 
\[s :\sigma\in S_{n}\mapsto s(\sigma). (g_{1},\dots,g_{n})=(g_{\sigma^{-1}(1)},\dots,g_{\sigma^{-1}(n)})\]
\begin{defi}
The wreath product between $G$ and $S_{n}$, $G\wr S_{n}$, is defined as the semi-direct product between $G^{n}$ and $S_{n}$, with $S_{n}$ acting on $G^{n}$ by the map $s$ above. More precisely, $$G\wr S_{n}=\lbrace ((g_{1},\dots,g_{n}),\sigma), g_{i}\in G,\sigma\in S_{n}\rbrace$$ with the product
\[((g_{1},\dots,g_{n}),\sigma)\cdot ((g'_{1},\dots,g'_{n}),\mu)=((g_{1}g'_{\sigma^{-1}(1)},\dots,g_ng'_{\sigma^{-1}(n)}),\sigma\mu).\]
\end{defi}

If $G$ is a compact group, $G\wr S_{n}$ is compact as well and thus there exists a Haar measure on $G\wr S_{n}$. It is direct to see that on $G\wr S_{n}$ is isomorph to $G\times \dots\times G\times S_{n}$ as a measure space and that the Haar measure on $G\wr S_n$ is given by $d\lambda_{G\wr S_{n}}=\bigotimes_{i} dg_{i}\otimes d\sigma$, where $d_{g}$ designates the Haar measure on $G$ and $d\sigma$ the normalized counting measure on $S_{n}$. 
If $\alpha : G\to U(V)$ is a unitary representation of $G$, then $G\wr S_{n}$ acts on $V^{\otimes n}$ via
\begin{align*}
\alpha^{n} ((g_{1},\dots,g_{n}),\sigma)(v_{1}\otimes\dots\otimes v_{n}):=
\alpha(g_{1})(v_{\sigma^{-1}(1)})\otimes\dots\otimes \alpha(g_{n})(v_{\sigma^{-1}(n)}).
\end{align*}

We will use the following notation in the sequel:
\begin{nota} If $\beta : G\to U(H)$ be a unitary representation of a compact group $G$, we denote :
\begin{enumerate}
\item[$\bullet$] $\chi_{\beta}$ denotes the character of $\beta$,
\item[$\bullet$] $F_{\beta}$ is the exponential generating serie of the moments of $\chi_{\beta}$ with respect to the Haar measure
\end{enumerate}
\end{nota}

The purpose is to describe the distribution of $\chi_{\alpha^{n}}$ under $d\lambda_{G\wr S_{n}}$,  when $\alpha$ is a represention of $G$. We will assume that $G\subset GL_{p}(\mathbb{R})$ for some $p\geq 1$. In particular $\chi_{\alpha^{n}}$ is real. The computations are similar in the complex setting; we just have to deal separately with the real and imaginary part of $\chi_{\alpha}$.\\
\begin{nota}
For each partition $\nu\in\mathcal{P}(k)$ with blocks $B_{1},\dots B_{r}$ and sequence of numbers $(c_{1},\dots,c_{n},\dots)$ of length greater than $k$ we write 
$$c_{\nu}=c_{\vert B_{1}\vert}c_{\vert B_{2}\vert}\dots c_{\vert B_{r}\vert}$$
with $\vert B_{i}\vert$ being the cardinal of the block $B_{i}$.
\end{nota}
\begin{prp}\label{clres}
The exponential serie of the moments of $\chi_{\alpha^n}$ is given by 

\begin{align*}
F_{\alpha^{n}}(x)=&\sum m_{\alpha^{n}}(k)\frac{x^{k}}{k!}
\end{align*}

with $$m_{\alpha^n}(k)=\sum_{\nu\in\mathcal{P}(k),l(\nu)\leq n}m_{\alpha}(\nu).$$
with $l(\pi)$ being the length of a partition $\pi$, that is the number of blocks of $\pi$.
\end{prp}

\begin{proof}
Let $t=((g_1,\dots,g_n),\sigma)\in G\wr S_n$, we have for $x>0$ small enough. Writing the action of $t$ through $\alpha^{n}$ in block matrices yields the following result
\begin{align*}
F_{\alpha^{n}}(x)=&\mathbb{E}_{G\wr S_{n}}(\exp(x\text{Tr}(\alpha^{n}(t)))=\int_{G\wr S_{n}}\exp\left(x\sum_{i\text{ fixed point of }\sigma} \text{Tr}(\alpha(g_{i}))\right)\prod dg_{i} d\sigma\\
=&\int_{S_{n}}\prod_{i\text{ fixed point of $\sigma$}} \left(\int_{G_{i}}\exp(x \times \text{Tr}(\alpha(g_{i})))dg_{i}\right) d\sigma\\
=&\int_{S_{n}}\prod_{i\text{ fixed point of $\sigma$}} F_{\alpha}(x)d\sigma\\
=&\int_{S_{n}}F_{\alpha}(x)^{\#\text{ fixed points of s}}d\sigma\\
=&\int_{S_{n}}\exp\left(\log(F_{\alpha}(x))\#\text{ fixed point of $\sigma$}\right)d\sigma.\\
\end{align*}
Considering $\log(F_{\alpha}(x))$ as fixed in the last integral yields the equality 
\begin{equation}
F_{\alpha^{n}}(x)=\log(F_{\alpha}(x))
\label{clstep1}
\end{equation}
where $F_{S_n}$ designates the exponential generating serie of the moments of the natural representation $S_n\hookrightarrow M_n(\mathbb{C})$. Now, we can exploit the general facts that 
\begin{equation}
F_{\beta}(x)=\sum m_{\beta}(k)\frac{x^{k}}{k!}
\label{cldef1}
\end{equation}
and 
\begin{equation}
\log F_{\beta}(x)=\sum c_{\beta}(k)\frac{x^{k}}{k!}
\label{cldef2}
\end{equation}
where $\left(m_{\beta}(k)\right)_{k\geq 1}$ are the moment of the law of $\chi_{\beta}$ and $\left(c_{\beta}(k)\right)_{k\geq 1}$ are the classical cumulants of this law. The latter is the only sequence of real numbers satisfying:
\begin{equation}
m_{\beta}(k)=\sum_{\pi\in\mathcal{P}(k)}c_{\beta}(\pi)
\label{cldef3}
\end{equation}
for all $k\geq 1$.
From the left-hand side of $\eqref{clstep1}$ and $\eqref{cldef1}$ we get
$$F_{\alpha^{n}}(x)=\sum_k m_{\alpha^{n}}(k)\frac{x^{k}}{k!}$$
 and from the right-hand of $\eqref{clstep1}$ with $\eqref{cldef2}$ we compute
\begin{align*}
F_{\alpha^{n}}(x)=&\sum_r m_{S_{n}}(r)\frac{\left(\sum c_{\alpha}(u)\frac{x^{u}}{u!}\right)^{r}}{r!}\\
=&\sum_k \frac{x^{k}}{k!} \sum_{r}\frac{m_{S_{n}}(r)}{r!}\sum_{\begin{subarray} au_{1}\times 1+\dots+u_{k}\times k=k\\ \sum u_{i}=r\end{subarray}}k!\frac{r!}{u_{1}!\dots u_{r}!}\left(\frac{c_{\alpha}(1)}{1!}\right)^{u_{1}}\dots \left(\frac{c_{\alpha}(k)}{k!}\right)^{u_{k}}\\
\end{align*}
The last equality above being due to the multinomial expansion. Hence, after identifying coefficients we obtain:
\begin{equation}m_{\alpha^{n}}(k)=\sum_rm_{S_n}(r)\sum_{\begin{subarray} au_{1}\times 1+\dots+u_{k}\times k=k\\\sum u_{i}=r\end{subarray}}k!\frac{1}{u_{1}!\dots u_{r}!}\left(\frac{c_{\alpha}(1)}{1!}\right)^{u_{1}}\dots \left(\frac{c_{\alpha}(k)}{k!}\right)^{u_{k}} \label{clstep2}\end{equation}
We say that a partition $p\in\mathcal{P}(k)$ is of type $(1^{u_{1}},\dots,k^{u_{r}})$, if it is a partition having $u_{1}$ blocks of cardinal $1$, $u_{2}$ of cardinal $2$ and so on. We know that the number of partitions of $\{1,\dots,k\}$ with type $(1^{u_{1}},\dots,k^{u_{r}})$ is exactly $$\frac{k!}{u_{1}!\dots u_{k}!}\frac{1}{1!^{u_{1}}\dots k!^{u_{k}}},$$ see e.g. page 22 in \cite{macdonald1998symmetric}. So we have by summing over every types of partition in \eqref{clstep2}:
\begin{equation}m_{\alpha^{n}}(k)=\sum_{r}m_{S_{n}}(r)\sum_{\pi\in\mathcal{P}(k),l(\pi)=r}c_{\alpha}(\pi)\label{clstep3}\end{equation}
Using the fact that (see \cite{BS09}, \cite{rosas2006symmetric})
\begin{equation}
m_{S_{n}}(r)=\#\lbrace\text{partitions of $\{1,\dots,r\}$ having at most $n$ blocks}\rbrace
\label{cldef4}\end{equation}
we can transform \eqref{clstep3} into
\begin{align*} 
m_{\alpha^{n}}(k)=&\sum_{r}\sum_{\substack{\nu\leq\mathbf{1}_{r}\\ l(\nu)\leq n}} \sum_{\substack{\pi\leq \mathbf{1}_{k}\\l(\pi)=r}}c_{\alpha}(\pi)=\sum_{r}\sum_{\substack{\pi\leq \mathbf{1}_{k}\\l(\pi)=r}}\sum_{\substack{\pi\leq\nu\leq\mathbf{1}_{k}\\l(\nu)\leq n}}c_{\alpha}(\pi)\\
=&\sum_{\pi\leq \mathbf{1}_{k}}\sum_{\substack{\pi\leq\nu\leq\mathbf{1}_{k}\\l(\nu)\leq n}}c_{\alpha}(\pi)
=\sum_{\substack{\nu\leq\mathbf{1}_{k}\\l(\nu)\leq n}}\sum_{\pi\leq\nu }c_{\alpha}(\pi)\\
=&\sum_{\substack{\nu\leq\mathbf{1}_{k}\\l(\nu)\leq n}}m_{\alpha}(\nu).\\
\end{align*}

\end{proof}

We can deduce from \ref{clres} the aymptotic law of $\chi_{\alpha^n}$ when $n$ goes to infinity :
\begin{crl} We have the convergence in moments
\[\chi_{\alpha^{n}}\underset{n\to \infty}{\longrightarrow}\mathcal{P}(\chi_{\alpha})\]
where $\mathcal{P}(\chi_{\alpha})$ is the compound Poisson law with parameter $1$ and original law $\alpha$.
\end{crl}
\begin{proof}
We have
\begin{align*}
m_{\alpha^{n}}(k)&=\sum_{\nu\leq\mathbf{1}_{k},l(\nu)\leq n}m_{\alpha}(\nu)\\
&\underset{n\to \infty}{\longrightarrow}\sum_{\nu\leq\mathbf{1}_{k}}m_{\alpha}(\nu)=m_{\mathcal{P}(\chi_{\alpha})}.
\end{align*}
\end{proof}
\begin{rque}
In the next section we will determine the intertwiner spaces for a free wreath product $\mathbb{G}\wr_{*}S_{N}^{+}$. The result and proofs can be easily adapted to get the same result in the classical case; one only need to use all partitions instead of non-crossing ones.
\end{rque}

\section{Intertwiner spaces in $\mathbb{G}\wr_*S_N^+$}\label{intfreec}
Let $\mathbb{G}=(C(\mathbb{G}),v)$ be a compact matrix quantum group of Kac type, generated by a unitary $v$ acting on $H$. In this section, the $C^*$-algebras associated with compact quantum groups are considered in their maximal versions. We consider a generating magic unitary $u$ of the free quantum permutation group $S_N^+$ acting on $\mathbb{C}^N$. We recall that the corepresentation $$\omega:=(\omega_{ijkl})_{1\le i,j\le N}^{1\le k,l\le d_{\mathbb{G}}}=(u_{ij}v_{kl}^{(i)})_{i,j,k,l}$$ acting on $W:=\mathbb{C}^N\otimes H$, is the generating matrix of the free wreath product quantum groups $\mathbb{G}\wr_*S_N^+$, see \cite{Bic04}. 

We set $\Rep(\mathbb{G})=\{\alpha\in I\}$ the set of equivalence classes of unitary finite dimensional (not necessarily irreducible) corepresentations of $\mathbb{G}$ and we denote by $H^\alpha=\langle Y_1^\alpha,\dots,Y_{d_{\alpha}}^\alpha\rangle$ the representation space of $\alpha$. We have a natural family of $\mathbb{G}\wr_*S_N^+$-representations, see the proof of Theorem 2.3 in \cite{Bic04}.

\begin{defi}\label{defralpha}
A family of unitary corepresentations of $\mathbb{G}\wr_*S_N^+$ is given by $$\{r(\alpha):=\left(u_{ij}\alpha_{kl}^{(i)}\right) : \alpha\in I\}.$$ Notice that $r(\alpha)$ acts on the vector space $\mathbb{C}^N\otimes H^{\alpha}$. These corepresentations will be called basic corepresentations for $\mathbb{G}\wr_*S_N^+$.

\end{defi}

We recall that we denote by $\{T_p\}_p$ the basis of $\Hom(u^{\otimes k},u^{\otimes l})$ with $p\in NC(k,l)$ for any $k,l\in\mathbb{N}$. For any $p\in NC(k,l)$ we make the convention that the points on top from left to right are weighted by the numbers $1,\dots,k$ and the points on bottom by the numbers $k+1,\dots,k+l$.

\begin{nota}\label{ordb} Let us chose an order on the blocks of the partitions $p\in NC(k,l)$. We write $p=\{B_1,\dots,B_r\}$, the block decomposition of $p$ with the following order: we fix $B_1$ the block containing $1$, $B_2$ the first block, if it exists, containing the smallest $1<i\le k+l$ with $i\notin B_1$, etc. 
\end{nota}


We want to describe the intertwiner spaces between tensor products of basic corepresentations of $\mathbb{G}\wr_*S_N^+$. These spaces will be described by linear maps associated with certain non-crossing partitions and with $\mathbb{G}$-morphisms. Indeed, let $[\alpha]:=(\alpha_1,\dots,\alpha_k)$ and $[\beta]:=(\beta_1,\dots,\beta_l)$ be tuples of $\mathbb{G}$-representations such that the points of $p$ are decorated by these corepresentations. This means that in each block $B_i$, certain corepresentations $\alpha^i_1,\dots,\beta^i_1,\dots,$ are attached to the upper and lower points respectively. 
We make the convention that if $k=0$, then the trivial corepresentation decorates the upper part of $p\in NC(0,l)$ and an analogue convention if $l=0$. The non-crossing partition describing intertwiners in $\mathbb{G}\wr_*S_N^+$ will also be such that their blocks are decorated by $\mathbb{G}$-morphisms. To be more precise, let us introduced some notation.


\begin{nota}\label{secdia} 
Let $p\in NC(k,l)$, $p$ given by its blocks denoted $B$. We will simplify the notation $B$ into $B$ when the context is clear. We denote:
\begin{enumerate} 
\item[$\bullet$] $B=U_B\cup L_B$ the upper and lower parts of each block $B$. 
\item[$\bullet$] We denote $H^{U_B}=\bigotimes_{i\in U_B}H^{\alpha_i}$ the tensor product of spaces $H^{\alpha_i}$, and similarly we denote $H^{L_B}=\bigotimes_{j\in L_B}H^{\beta_j}$.
\item[$\bullet$] We denote $\alpha(U_B)=\bigotimes_{i\in U_B}\alpha_i$ the tensor product of corepresentations $\alpha_i$ and similarly we denote $\beta(L_B)=\bigotimes_{j\in L_B}\beta_j$.
\end{enumerate}

\end{nota}

Furthermore, we assume that ``attached" to each block $B$ there is a $\mathbb{G}$-morphism 
\begin{align}\label{lonal}
S_B=\alpha(U_B)\to\beta(L_B)\in \mathcal{B}(H^{U_B},H^{L_B})
\end{align}
 and we put

\begin{align}\label{longalpha}
S=\bigotimes_BS_B: \bigotimes_B\alpha(U_B)\to\bigotimes_B\beta(L_B)
\end{align}
with the order on the blocks we gave above.
We say that the blocks of $p$ are \textit{decorated} by $[S]=(S_1,\dots,S_r)$ where $r$ is the number of blocks in $p$. 
\begin{defi}\label{admis}
We say that the partition $p$ decorated by representations $[\alpha],[\beta]$ is admissible if $\forall B\in p, \Hom_{\mathbb{G}}(\alpha(U_B);\alpha(L_B))\ne0$. 
\end{defi}

Therefore, we can consider  
\begin{align*}
T_p\otimes S \in \mathcal{B}\left((\mathbb{C}^N)^{\otimes k}\otimes\bigotimes_{B} H^{U_B} ; (\mathbb{C}^N)^{\otimes l}\otimes\bigotimes_{B}^r H^{L_B}\right).
\end{align*}

\begin{rque}\label{theyrefree}
Notice that if the $\mathbb{G}$-morphisms in $\mathcal{B}(H^{U_B},H^{L_B})$ run over a basis of intertwiners $\alpha(U_B)\to\beta(L_B)$ then the family $(T_p\otimes S)_{p,S}$ is free.
\end{rque}

We shall twist this linear map to obtain a morphism 
\begin{align*}
\widetilde{T_p\otimes S}\in \Hom_{\mathbb{G}\wr_* S_N^+}(r(\alpha_1)\otimes\dots\otimes r(\alpha_k), r(\beta_1)\otimes\dots\otimes r(\beta_l)).
\end{align*}

\begin{nota}\label{cantwist}
Let $p\in NC(k,l)$ decorated by $\mathbb{G}$-representations $[\alpha],[\beta]$ and morphisms $[S]$ as in the above notation. One can consider a unitary $t_p^U$ acting on vectors $x_i\in \mathbb{C}^N$, $y_i\in H^{\alpha_i}$, $i=1,\dots,k$
\begin{align*}
t_p^{U} : (\mathbb{C}^N\otimes H^{\alpha_1})\otimes\dots\otimes (\mathbb{C}^N\otimes H^{\alpha_k})&\to (\mathbb{C}^N)^{\otimes k}\otimes\bigotimes_{B} H^{U_B},\\
\bigotimes_{i=1}^k(x_i\otimes y_i)&\mapsto\bigotimes_{i=1}^kx_i\otimes\bigotimes_{B}\bigotimes_{i'\in U_B}y_{i'}
\end{align*}
and a unitary $t_p^L$ acting on vectors $x_j\in \mathbb{C}^N$, $y_j\in H^{\beta_j}$, $j=1,\dots,l$
\begin{align*}
t_p^{L} :  (\mathbb{C}^N\otimes H^{\beta_1})\otimes\dots\otimes (\mathbb{C}^N\otimes H^{\beta_l}) &\to (\mathbb{C}^N)^{\otimes l}\otimes\bigotimes_{B} H^{L_B},\\
\bigotimes_{j=1}^l(x_j\otimes y_j) &\mapsto \bigotimes_{j=1}^lx_j\otimes\bigotimes_{B}\bigotimes_{j'\in L_B}y_{j'}.
\end{align*}
We denote 
\begin{align*}
U^{p,S}&:=(t_L^{p})^*\circ(T_p\otimes S)\circ t_U^p\\
&\in \mathcal{B}\left((\mathbb{C}^N\otimes H^{\alpha_1})\otimes\dots\otimes (\mathbb{C}^N\otimes H^{\alpha_k}),(\mathbb{C}^N\otimes H^{\beta_1})\otimes\dots\otimes (\mathbb{C}^N\otimes H^{\beta_l}) \right).
\end{align*}
\end{nota}


We can now prove the following result:

\begin{thm}\label{bigthm}
Let $\mathbb{G}=(C(\mathbb{G}),\Delta)$ be a compact quantum group of Kac type. Let $\alpha_1,\dots,\alpha_k$ and $\beta_1,\dots,\beta_l$ be finite dimensional corepresentations in $\Rep(\mathbb{G})$. We set $[\alpha]=(\alpha_1,\dots,\alpha_k)$ and $[\beta]=(\beta_1,\dots,\beta_l)$. Then

\begin{align}\label{ncint}
\Hom_{\mathbb{G}\wr_* S_N^+}(r(\alpha_1)\otimes\dots&\otimes r(\alpha_k) ; r(\beta_1)\otimes\dots\otimes r(\beta_l))\\
\label{ncint2}&=\emph{span}\{U^{p,S} : p\in NC_{\mathbb{G}}([\alpha],[\beta]), S \emph{ as below}\}
\end{align}

where $U^{p,S}=(t_L^p)^*\circ(T_p\otimes S)\circ t_U^p$ with
\begin{enumerate}
\item[$\bullet$] the isomorphisms $t_U^p, t_L^p$ defined in Notation \ref{cantwist},
\item[$\bullet$]  $NC_{\mathbb{G}}([\alpha],[\beta])$ consists of non-crossing partitions in $NC(k,l)$ decorated with corepresentations $[\alpha], [\beta]$ on the upper and lower points respectively,
\item[$\bullet$] $S=\bigotimes_BS_B: \bigotimes_B\alpha(U_B)\to\bigotimes_B\beta(L_B)$ as in (\ref{longalpha}), where the $\mathbb{G}$-morphisms in $\mathcal{B}(U_B,L_B)$ which decorate the blocks $B\in p$ run over intertwiners $\alpha(U_B)\to\beta(L_B)$.
\end{enumerate}
\end{thm}

\begin{proof}
We first prove that $$U^{p,S}\in \Hom_{\mathbb{G}\wr_* S_N^+}(r(\alpha_1)\otimes\dots\otimes r(\alpha_k) ; r(\beta_1)\otimes\dots\otimes r(\beta_l))$$ that is the inclusion of the right hand space (\ref{ncint2}) in the left hand space (\ref{ncint}).

The Frobenius reciprocity for $C^*$-tensor categories with conjugates provide the following isomorphisms: 
\begin{align*}
\Hom_{\mathbb{G}\wr_* S_N^+}(r(\alpha_1)\otimes\dots\otimes r(\alpha_k)&; r(\beta_1)\otimes\dots\otimes r(\beta_l))\\
&\simeq \Hom_{\mathbb{G}\wr_* S_N^+}\left(1; \overline{r(\alpha_1)}\otimes\dots\otimes \overline{r(\alpha_k)}\otimes r(\beta_1)\otimes\dots\otimes r(\beta_l)\right)\\
&\simeq \Hom_{\mathbb{G}\wr_* S_N^+}\left(1 ; r(\bar\alpha_1)\otimes\dots\otimes r(\bar\alpha_k)\otimes r(\beta_1)\otimes\dots\otimes r(\beta_l)\right),
\end{align*}
$$\Hom_{\mathbb{G}}(\alpha_1\otimes\dots\otimes\alpha_k;\beta_1\otimes\dots\otimes\beta_l)\simeq \Hom_{\mathbb{G}}(1;\bar \alpha_1\otimes\dots\otimes\bar\alpha_k\otimes\beta_1\otimes\dots\otimes\beta_l).$$
Hence, one can restrict to prove that
\begin{align}\label{restrict}
t_L^{p}(T_p\otimes\xi)\in \Hom_{\mathbb{G}\wr_* S_N^+}(1;r(\alpha_1)\otimes\dots\otimes r(\alpha_k))
\end{align}
for all $k\in \mathbb{N}$, $p\in NC(k)$ and all fixed vectors $$\xi=\bigotimes_{B} \xi(L_B) : \mathbb{C}\to\bigotimes_{B} H^{L_B}.$$
It is enough to prove (\ref{restrict}) for the one block partition $1_k$ since one can recover any $p\in NC$ by tensor products and compositions of partitions $1_k$ and $\id$.

We now fix $(e_i)_{i=1}^N$ a basis of $\mathbb{C}^N$ and $(Y_j^{\alpha})_{j=1}^{d_{\alpha}}$ a basis of $H^{\alpha}$, for any $\alpha\in \Rep(\mathbb{G})$. Proving $t_L^{1_k}(T_p\otimes \xi)\in \Hom_{\mathbb{G}\wr_* S_N^+}(1;r(\alpha_1)\otimes\dots\otimes r(\alpha_k))$ for some 
$$\xi=\sum_{[j]}\lambda^k_{[j]}Y_{j_1}^{\alpha_1}\otimes\dots\otimes Y_{j_k}^{\alpha_k}\in \Hom_{\mathbb{G}}(1;\alpha_1\otimes\dots\otimes \alpha_k)$$ then follows from the following computation. 
We put $T_{\xi}^{1_k}:=t_L^{1_k}(T_p\otimes\xi)$ and we then have

$$T_{\xi}^{1_k}\equiv\sum_{i,[j]}\lambda_{[j]}^k(e_{i}\otimes Y_{j_{1}}^{\alpha_{1}})\otimes\dots\otimes(e_{i}\otimes Y_{j_{k}}^{\alpha_{k}})$$ so that 
$$r_{\alpha_1}\otimes\dots\otimes r_{\alpha_k}(T_{\xi}^{1_k}\otimes 1)=\sum_{i,[j]}\lambda_{[j]}^k\sum_{[r],[s]}(e_{s_1}\otimes Y_{r_1}^{\alpha_{1}})\otimes\dots\otimes (e_{s_k}\otimes Y_{r_k}^{\alpha_{k}})\otimes \left(u_{s_1i}(\alpha_1)_{r_1j_1}^{(s_1)}\dots u_{s_ki}(\alpha_k)_{r_kj_k}^{(s_k)}\right).$$
But the magic unitary $u$ satisfies for all $s,t,$ $u_{si}u_{ti}=\delta_{st}u_{si}, \sum_iu_{si}=1$ and then combining this with the commuting relations in the free wreath product $C(\mathbb{G})*_wC(S_N^+)$, we get

\begin{align}\label{bigop}
r_{\alpha_1}\otimes\dots\otimes r_{\alpha_k}(T_{\xi}^{1_k}\otimes 1)&=\sum_{[j]}\lambda_{[j]}^k\sum_{[r],s_1}(e_{s_1}\otimes Y_{r_1}^{\alpha_{1}})\otimes\dots\otimes (e_{s_1}\otimes Y_{r_k}^{\alpha_{k}})\otimes \left((\alpha_1)_{r_1j_1}^{(s_1)}\dots (\alpha_k)_{r_kj_k}^{(s_1)}1\right)\nonumber\\
&=\sum_{s_1}\sum_{[j]}\lambda_{[j]}^k\sum_{[r]}(e_{s_1}\otimes Y_{r_1}^{\alpha_{1}})\otimes\dots\otimes (e_{s_1}\otimes Y_{r_k}^{\alpha_{k}})\otimes \left((\alpha_1)_{r_1j_1}^{(s_1)}\dots (\alpha_k)_{r_kj_k}^{(s_1)}\right).
\end{align}

Now, since 
\begin{align}\label{fixx}
\xi=\sum_{[j]}\lambda^k_{[j]}Y_{j_1}^{\alpha_1}\otimes\dots\otimes Y_{j_k}^{\alpha_k}\in \Hom_{\mathbb{G}}(1;\alpha_1\otimes\dots\otimes \alpha_k),
\end{align}

we get applying $(t_L^{1_k})^{-1}$ in (\ref{bigop}), using (\ref{fixx}) and applying once more $t_L^{1_k}$,
\begin{align*}
r_{\alpha_1}\otimes\dots\otimes r_{\alpha_k}(T_{\xi_{1_k}}^{1_k}\otimes 1)&=\sum_{[r],s_1}\lambda_{[r]}^k(e_{s_1}\otimes Y_{r_1}^{\alpha_{1}})\otimes\dots\otimes(e_{s_1}\otimes Y_{r_k}^{\alpha_{k}})\otimes 1\\
&=T_{\xi}^{1_k}\otimes 1.
\end{align*}




We now define $\mathcal{T}$ as the rigid monoidal $C^*$-tensor category generated by the collection of $\mathbb{G}\wr_*S_N^+$-intertwiners spaces $\text{span}\left\{U^{p,S} : p,S \text{ as in } (\ref{ncint2}) \right\}$ between objects indexed by families $[\alpha]$ of $\mathbb{G}$-representations, 
If one applies Woronowicz's Tannaka-Krein duality to this category $\mathcal{T}$, we get a compact matrix quantum group $(\mathbb{H},\Omega)$ generated by a unitary $\Omega$ corresponding to $r(v)\in \mathcal{B}(\mathbb{C}^N\otimes H)\otimes C(\mathbb{H})$ and a family of corepresentations $(R_{\alpha_{i}})_{i\in I}$ such that 
\begin{align*}
\Hom_{\mathbb{H}}(R_{\alpha_{1}}\otimes\dots&\otimes R_{\alpha_{k}};R_{\beta_{1}}\otimes\dots\otimes R_{\beta_{l}})\\
&=\text{span}\left\{U^{p,S} : p,S \text{ as in } (\ref{ncint2})\right\},
\end{align*}
with $p\in NC(k,l)$, $S : \bigotimes_B\alpha(U_B)\to\bigotimes_B\beta(L_B)$, $[\alpha]=(\alpha_{1},\dots,\alpha_{k})$, $\beta=(\beta_{1},\dots,\beta_{l})$.

We proved above that $$U^{p,S}\in \Hom_{\mathbb{G}\wr_*S_N^+}(r(\alpha_{1})\otimes\dots\otimes r(\alpha_{k});r(\beta_{1})\otimes\dots\otimes r(\beta_{l})).$$ In particular, there is by universality a (surjective) morphism $$\pi_1 :C(\mathbb{H})\to C(\mathbb{G}\wr_*S_N^+),\ \ \Omega_{ijkl}\mapsto\omega_{ijkl}.$$ To prove the theorem we shall construct a surjective morphism $\pi_2 : C(\mathbb{G}\wr_*S_N^+)\to C(\mathbb{H})$ such that $$\pi_1\circ\pi_2=\id=\pi_2\circ\pi_1.$$

We define the following elements in $C(\mathbb{H})$ 
\begin{align}\label{defwr}
V_{kl}^{(i)}:=\sum_{j}\Omega_{ijkl}\ \ \text{ and }\ \ U_{ijk}=\sum_{l}\Omega_{ijkl}\Omega_{ijkl}^*.
\end{align}

We shall prove that the generating relations in $C(\mathbb{G}\wr_*S_N^+)$ are also satisfied by the elements $V_{kl}^{(i)}$ and $U_{ijk}$ in $C(\mathbb{H})$.

Since the generating matrix $v$ of $\mathbb{G}$ is unitary, we get that $\xi=\sum_bY_b\otimes \bar Y_b$ is a fixed vector of $v\otimes\bar v$ and thus $\xi\otimes \xi\in \Hom(1;v\otimes\bar v\otimes v\otimes\bar v)\simeq \Hom(v;v\otimes\bar v\otimes v)$. Via this isomorphism, we identify $\xi\otimes\xi$ with $Y\mapsto \sum_c Y_c\otimes \overline{Y_c}\otimes Y$. 

We then have an intertwiner $T:=t_L^p(T_p\otimes \xi\otimes\xi)\in \Hom(\Omega;\Omega\otimes\bar \Omega\otimes \Omega)\subset\mathcal{T}$ with 

\setlength{\unitlength}{0.5cm}
$$
p=\left\{\ \begin{picture}(2.1,2)\thicklines
\put(1,0.3){\line(0,1){1.3}}
\put(0,0.3){\line(1,0){2.028}}
\put(0.029,-1){\line(0,1){1.3}}
\put(1,-1){\line(0,1){1.3}}
\put(2,-1){\line(0,1){1.3}}
\end{picture}\ \right\}
$$

i.e. with Notation \ref{secdia} and making plain the $\mathbb{G}$-morphisms on the block $p$

\setlength{\unitlength}{0.5cm}
$$
P=\left\{\ \begin{picture}(1.8,2)\thicklines
\put(1,0.3){\line(0,1){1.3}}
\put(0.7,1.7){$v$}
\put(0.7,0.1){\line(0,1){1.5}}
\put(0.7,0.126){\line(1,0){1}}
\put(0,0.3){\line(1,0){2.028}}
\put(-0.3,-0.1){\line(1,0){1}}
\put(0.029,-1){\line(0,1){1.3}}
\put(-0.3,-1.5){$v$}
\put(-0.3,-1){\line(0,1){0.926}}
\put(1,-1){\line(0,1){1.3}}
\put(0.7,-1.5){$\bar v$}
\put(0.7,-1){\line(0,1){0.926}}
\put(2,-1){\line(0,1){1.3}}
\put(1.7,-1.5){$v$}
\put(1.7,-1){\line(0,1){1.155}}
\end{picture}\ \right\}
$$

that is
$$T(Y\otimes e_a)=\sum_{c}(e_a\otimes Y_c)\otimes\overline{(e_a\otimes Y_c)}\otimes( e_a\otimes Y).$$ We obtain for all $a=1,\dots,N$ and $b=1,\dots,d_{\mathbb{G}}$:
\begin{align*}
&\sum_{[i],[k],c}(e_{i_1}\otimes Y_{k_1})\otimes(e_{i_2}\otimes Y_{k_2})\otimes(e_{i_3}\otimes Y_{k_3})\otimes \Omega_{i_1ak_1c}\ \Omega_{i_2ak_2c}^*\ \Omega_{i_3ak_3b}\\
&=\sum_{i,k,r}(e_i\otimes Y_r)\otimes(e_i\otimes Y_r)\otimes (e_i\otimes Y_k)\otimes \Omega_{iakb}
\end{align*}

so that for all $[i]\in \{1,\dots,N\}^3$, $[k]\in\{1,\dots,d_{\mathbb{G}}\}^2$, $a\in\{1,\dots,N\}$ and $b\in\{1,\dots,d_{\mathbb{G}}\}$:

\begin{align}\label{projt}
\left(\sum_{c} \Omega_{i_1ak_1c}\ \Omega_{i_2ak_2c}^*\right) \Omega_{i_3ak_3b}
&=\delta_{i_1,i_2,i_3}\delta_{k_1,k_2}\Omega_{i_3ak_3b}.
\end{align}

and taking adjoints:
\begin{align}\label{projt1}
\Omega_{i_3ak_3b}^*\left(\sum_{c} \Omega_{i_2ak_2c}\ \Omega_{i_1ak_1c}^*\right)
&=\delta_{i_1,i_2,i_3}\delta_{k_1,k_2}\Omega_{i_3ak_3b}^*.
\end{align}

Considering now 

\setlength{\unitlength}{0.5cm}
$$
P'=\left\{\ \begin{picture}(2.1,2)\thicklines
\put(1,0.3){\line(0,1){1.3}}
\put(0.7,1.75){$v$}
\put(0.7,0.1){\line(0,1){1.5}}
\put(-0.3,0.126){\line(1,0){1}}
\put(0,0.3){\line(1,0){2.028}}
\put(0.7,-0.1){\line(1,0){1}}
\put(0.029,-1){\line(0,1){1.3}}
\put(-0.35,-1.5){$v$}
\put(0.7,-1){\line(0,1){0.926}}
\put(0.7,-1.5){$v$}
\put(1,-1){\line(0,1){1.3}}
\put(1.7,-1){\line(0,1){0.926}}
\put(1.7,-1.5){$\bar v$}
\put(2,-1){\line(0,1){1.3}}
\put(-0.3,-1){\line(0,1){1.155}}
\end{picture}\ \right\},
$$
we can get the same way, for all $[i]\in \{1,\dots,N\}^3$, $[k]\in\{1,\dots,d_{\mathbb{G}}\}^2$, $a\in\{1,\dots,N\}$ and $b\in\{1,\dots,d_{\mathbb{G}}\}$, using $t_L^p(T_{p'}\otimes \xi\otimes\xi)\in \Hom(\Omega;\Omega\otimes\Omega\otimes\bar\Omega)\subset\mathcal{T}$,
\begin{align}\label{projt2}
\Omega_{i_3ak_3b}\left(\sum_{c} \Omega_{i_1ak_1c}\ \Omega_{i_2ak_2c}^*\right)
&=\delta_{i_1,i_2,i_3}\delta_{k_1,k_2}\Omega_{i_3ak_3b},
\end{align}
and taking adjoints:
\begin{align}\label{projt3}
\left(\sum_{c} \Omega_{i_2ak_2c}\ \Omega_{i_1ak_1c}^*\right)\Omega_{i_3ak_3b}^*
&=\delta_{i_1,i_2,i_3}\delta_{k_1,k_2}\Omega_{i_3ak_3b}^*.
\end{align}

We shall obtain from (\ref{projt}), (\ref{projt1}), (\ref{projt2}), (\ref{projt3}) all the necessary relations in $C(\mathbb{H})$ to build back the free wreath product $\mathbb{G}\wr_*S_N^+$. 

From these relations, we see in particular that the elements $U_{ijk}=\sum_{c}\Omega_{ijkc}\Omega_{ijkc}^*$ do not depend on $k$ since
\begin{align*}
U_{ijk}U_{ijk'}&=\sum_{c,d}\Omega_{ijkc}\Omega_{ijkc}^*\Omega_{ijk'd}\Omega_{ijk'd}^*\\
&=\sum_d\left(\sum_c\Omega_{ijkc}\Omega_{ijkc}^*\Omega_{ijk'd}\right)\Omega_{ijk'd}^*\\
&=\sum_d\Omega_{ijk'd}\Omega_{ijk'd}^*=U_{ijk'} &&\text{ (by (\ref{projt}))},\\
\end{align*}
and similarly $U_{ijk}U_{ijk'}=U_{ijk}$, using (\ref{projt2}). We then obtain $U_{ijk}=U_{ijk'}$. We fix $k$ and set $U_{ij}:=U_{ijk}$. Notice that the case $k=k'$ above shows that $U_{ij}$ is an orthogonal projection (the relation $U_{ij}^*=U_{ij}$ is clear). 
In fact, the matrix $(U_{ij})$ is a magic unitary, since it is a unitary whose entries are orthogonal projections.
We now prove that for all $i=1,\dots,N$ and all $\epsilon_j,\epsilon_k'\in\{1,*\},$
$$\Hom_{\mathbb{G}}\left(v^{\epsilon_1}\otimes\dots\otimes v^{\epsilon_k};v^{\epsilon'_1}\otimes\dots\otimes v^{\epsilon'_l})\subset \Hom_{\mathbb{H}_i}(V^{(i)\epsilon_1}\otimes\dots\otimes V^{(i)\epsilon_k};V^{(i)\epsilon'_1}\otimes\dots\otimes V^{(i)\epsilon'_l}\right),$$
where $\mathbb{H}_{i}$ is the compact matrix quantum groups whose underlying Woronowicz-$C^*$-algebra is generated by the coefficients of $V^{(i)}$.
By Frobenius reciprocity, it is enough to prove that any fixed vector in $\mathbb{G}$ is fixed in $\mathbb{H}_i$.

If $\xi_{k}=\sum_{[j]}\lambda_{[j]}Y_{j_1}\otimes\dots\otimes Y_{j_k}\in \Hom(1;v^{\epsilon_1}\otimes\dots\otimes v^{\epsilon_k})$, we have:
\begin{align*}
\sum_{[r][j]}\lambda_{[j]}Y_{r_1}\otimes\dots\otimes Y_{r_k}\otimes v_{r_1j_1}^{\epsilon_1}\dots v_{r_kj_k}^{\epsilon_k}=\sum_{[r]}\lambda_{[r]}Y_{r_1}\otimes\dots\otimes Y_{r_k}\otimes 1,
\end{align*}

i.e. $\forall [r]\in \{1,\dots,d_{\mathbb{G}}\}^k$, we have the relations in $C(\mathbb{G})$:

\begin{align}\label{transfer}
\sum_{[j]} \lambda_{[j]}v_{r_1j_1}^{\epsilon_1}\dots v_{r_kj_k}^{\epsilon_k}=\lambda_{[r]}.
\end{align}

Now, we use the morphism $(t_L^p)^*\circ(T_p\otimes \xi_k)\in \mathcal{T}$, with $p=1_k\in NC(k)$ i.e.  
\begin{align*}
(t_L^p)^*\circ(\xi_k\otimes T_p)&=\sum_{i[j]}\lambda_{[j]}(e_{i}\otimes Y_{j_1})\otimes\dots\otimes(e_{i}\otimes Y_{j_k})\\
&\in \Hom(1;\Omega^{\epsilon_1}\otimes\dots\otimes\Omega^{\epsilon_k})\subset \mathcal{T}.
\end{align*}

We get
\begin{align}
\sum_{[r][t]}(e_{r_1}\otimes Y_{t_1})&\otimes\dots\otimes(e_{r_k}\otimes Y_{t_k})\otimes\sum_{i[j]}\lambda_{[j]}\Omega_{r_1it_1j_1}^{\epsilon_1}\dots\Omega_{r_kit_kj_k}^{\epsilon_k}\label{relV}\\
=&\sum_{r[t]}\lambda_{[t]}(e_{r}\otimes Y_{t_1})\otimes\dots\otimes(e_{r}\otimes Y_{t_k})\otimes1\label{relV1}.
\end{align}

Notice that the relations (\ref{projt}), (\ref{projt1}), (\ref{projt2}), (\ref{projt3}) yield for $\epsilon=1,*$ and all $i,j,k,l$:

\begin{align}\label{comOmU}
U_{ij}\Omega_{ijkl}^{\epsilon}=\Omega_{ijkl}^{\epsilon}=\Omega_{ijkl}^{\epsilon}U_{ij}.
\end{align}

Then using these commuting relations and the fact that $(U_{ij})$ is a magic unitary, we get from (\ref{relV}):
\begin{align*}
\sum_{[r][t]}(e_{r_1}\otimes Y_{t_1})&\otimes\dots\otimes(e_{r_k}\otimes Y_{t_k})\otimes\sum_{i[j]}\lambda_{[j]}\Omega_{r_1it_1j_1}^{\epsilon_1}\dots\Omega_{r_kit_kj_k}^{\epsilon_k}\\
&=\sum_{[r][t]}(e_{r_1}\otimes Y_{t_1})\otimes\dots\otimes(e_{r_k}\otimes Y_{t_k})\otimes\sum_{i[j]}\lambda_{[j]}(\Omega_{r_1it_1j_1}^{\epsilon_1}U_{r_1i})\dots(\Omega_{r_kit_kj_k}^{\epsilon_k}U_{r_ki})\\
&=\sum_{r_1[t]}(e_{r_1}\otimes Y_{t_1})\otimes\dots\otimes(e_{r_1}\otimes Y_{t_k})\otimes\sum_{i[j]}\lambda_{[j]}(\Omega_{r_1it_1j_1}^{\epsilon_1}\dots\Omega_{r_1it_kj_k}^{\epsilon_k})(U_{r_1i}\dots U_{r_1i})\\
&=\sum_{r_1[t]}(e_{r_1}\otimes Y_{t_1})\otimes\dots\otimes(e_{r_1}\otimes Y_{t_k})\otimes\sum_{[i][j]}\lambda_{[j]}(\Omega_{r_1it_1j_1}^{\epsilon_1}\dots\Omega_{r_1it_kj_k}^{\epsilon_k})(U_{r_1i_1}\dots U_{r_1i_k})\\
&=\sum_{r_1[t]}(e_{r_1}\otimes Y_{t_1})\otimes\dots\otimes(e_{r_1}\otimes Y_{t_k})\otimes\sum_{[i][j]}\lambda_{[j]}(\Omega_{r_1i_1t_1j_1}^{\epsilon_1}U_{r_1i_1})\dots(\Omega_{r_1i_kt_kj_k}^{\epsilon_k}U_{r_1i_k})\\
&=\sum_{r_1[t]}(e_{r_1}\otimes Y_{t_1})\otimes\dots\otimes(e_{r_1}\otimes Y_{t_k})\otimes\sum_{[j]}\lambda_{[j]} V^{(r_1)\epsilon_1}_{t_1j_1}\dots V^{(r_1)\epsilon_k}_{t_kj_k}.\\
\end{align*}

Hence with (\ref{relV1}), we obtain $\forall [t]\in \{1,\dots,d_{\mathbb{G}}\}^k$

$$\sum_{[j]}\lambda_{[j]}V_{t_1j_1}^{(r_1)\epsilon_1}\dots V_{t_kj_k}^{(r_1)\epsilon_k}=\lambda_{[t]},$$

so that $\xi_{k}=\sum_{[j]}\lambda_{[j]}Y_{j_1}\otimes\dots\otimes Y_{j_k}\in \Hom_{\mathbb{H}_r}\left(1;V^{(r)\epsilon_1}\otimes\dots\otimes V^{(r)\epsilon_k}\right)$ for all $r=1,\dots,N$.

Then, we obtain that $\Rep(\mathbb{G})\subset \Rep(\mathbb{H}_i)\subset \Rep(\mathbb{H})$ as full sub-categories. Woronowicz's Tannaka-Krein duality theorem then implies that for all $i=1,\dots,N$ there exists a morphism $$\pi_i : C(\mathbb{G})\to C(\mathbb{H}_i)\subset C(\mathbb{H})$$ sending $v$ to $V^{(i)}$. 


Now, we prove that $V_{kl}^{(i)}U_{ij}=\Omega_{ijkl}=U_{ij}V_{kl}^{(i)}$. This follows from (\ref{comOmU}): 
\begin{align*}
V_{kl}^{(i)}U_{ij}=\sum_{J}\Omega_{iJkl}U_{ij}=\Omega_{ijkl}U_{ij}=\Omega_{ijkl}
\end{align*}
and similarly
\begin{align*}
U_{ij}V_{kl}^{(i)}=\Omega_{ijkl}.
\end{align*}
It follows from what we have proved above that there exist morphisms
\begin{enumerate}
\item[$\bullet$] $\pi_i : C(\mathbb{G})\to C(\mathbb{H}_i)$ such that $\pi_i\left(v_{kl}^{(i)}\right)=V_{kl}^{(i)}$, for all $i=1,\dots,N$,
\item[$\bullet$] $\pi_{N+1} : C(S_N^+)\to C(\mathbb{H})$ such that $\pi_{N+1}(u_{ij})=U_{ij}$.
\end{enumerate}
Thanks to the commuting relations we obtained above, these morphisms induce a morphism $\pi_2 : C(\mathbb{G}\wr_*S_N^+)\to C(\mathbb{H})$, such that $\pi_2\left(v_{kl}^{(i)}u_{ij}\right)=V_{kl}^{(i)}U_{ij}$. By construction, we then get $\pi_1\circ\pi_2=\id=\pi_2\circ\pi_1$ and the proof is complete.


\end{proof}

\begin{rque}\label{classical}
In the case where $\mathbb{G}$ is the dual of a discrete (classical) group $\mathbb{G}=\widehat{\Gamma}$, we recover the results of \cite{BV09} and \cite{Lem13b}. Indeed, in this case, the irreducible corepresentations of $\mathbb{G}=(C^*(\Gamma),\Delta)$ are the one-dimensional group like corepresentations $\Delta(g)=g\otimes g, g\in\Gamma$, the trivial one is the neutral element $e$ and the tensor product of two irreducible corepresentations is their product in $\Gamma$. 
Any morphism $$S_{[g],[h]} : \mathbb{C}\simeq\mathbb{C}^{\otimes k}\to\mathbb{C}^{\otimes l}\simeq\mathbb{C}, g_1\dots g_k\to h_1\dots h_l$$ is determined by the image of $1\in\mathbb{C}$ and the tensor products $S_{[g],[h]}\otimes T_p$ are scalar multiplication of the linear maps $T_p$. The space $$Hom_{\widehat{\Gamma}\wr_*S_N^+}(r(g_1)\otimes\dots\otimes r(g_k);r(h_1)\otimes\dots\otimes r(h_l))$$ is generated by the maps $T_p$ where $p\in NC(k,l)$ is an admissible diagram in $NC_{\widehat{\Gamma}}$ as in Definition \ref{admis}. In this setting, $p$ is admissible if $p\in NC(k,l)$ has the additional rules that if one decorates the points of $p$ by the elements $g_i, h_j$ then in each block, the product on top is equal to the product on bottom in $\Gamma$.
\end{rque}

In the sequel, we denote by $1_{\mathbb{G}}$ the trivial $\mathbb{G}$-representation and simply by $1$ the one of $\mathbb{G}\wr_*S_N^+$.

\begin{crl}\label{corbas} Let $N\ge4$, then:
\begin{enumerate}
\item For all $\alpha_1,\dots,\alpha_k,\beta_1,\dots,\beta_l\in \text{Rep}(\mathbb{G})$, we have 
\begin{align*}
\dim\Hom_{\mathbb{G}\wr_*S_N^+}&(r(\alpha_1)\otimes\dots\otimes r(\alpha_k);r(\beta_1)\otimes\dots\otimes r(\beta_l))\\
&=\sum_{p\in NC_{\mathbb{G}}([\alpha],[\beta])}\prod_{B\in p}\dim\Hom_{\mathbb{G}}(\alpha(U_B),\beta(L_B)).
\end{align*}
\item If $\alpha\in \Irr(\mathbb{G})$ is non-equivalent to $1_{\mathbb{G}}$ then $r(\alpha)$ is an irreducible $\mathbb{G}\wr_*S_N^+$-representation.
\item $r(1_{\mathbb{G}})=(u_{ij})=1\oplus \omega(1_{\mathbb{G}})$ for some $\omega(1_{\mathbb{G}})\in \Irr(\mathbb{G}\wr_*S_N^+)$.
\item Denoting $\omega(\alpha):=r(\alpha)\ominus\delta_{\alpha,1_{\mathbb{G}}}1$ then $(\omega(\alpha))_{\alpha\in \Irr(\mathbb{G})}$ is a family of pairwise non-equivalent $\mathbb{G}\wr_*S_N^+$-irreducible corepresentations.
\end{enumerate}
\end{crl}
\begin{proof}
We use Theorem \ref{bigthm} and the independence of the linear maps $$T_p\in \mathcal{B}((\mathbb{C}^N)^{\otimes k},(\mathbb{C}^N)^{\otimes l}),\ p\in NC(k,l)$$ for all $N\ge4$. The first assertion follows from this linear independence of the maps $T_p$. Indeed, we have
\begin{align*}
\Hom_{\mathbb{G}\wr_*S_N^+}&(r(\alpha_1)\otimes\dots\otimes r(\alpha_k);r(\beta_1)\otimes\dots\otimes r(\beta_l))\\
&=\bigoplus_{p\in NC_{\mathbb{G}}([\alpha],[\beta])}\text{span}\left\{U^{p,S} : \forall B, S_B\in\Hom_{\mathbb{G}}(\alpha(U_B),\beta(L_B))\right\}
\end{align*}
and the first assertion follows by computing the dimension of the spaces of each side.

Now we prove simultaneously the last three relations. For $\alpha,\beta\in \Irr(\mathbb{G})$, the intertwiner space $$\Hom_{\mathscr{\mathbb{G}}\wr_*S_N^+}(r(\alpha),r(\beta))$$ is encoded by the following candidate diagrams:

\setlength{\unitlength}{0.5cm}
$$
p_1=\left\{\ \begin{picture}(0.2,1.5)\thicklines
\put(0.1,-0.25){\line(0,1){0.9}}
\put(-0.1,0.9){$\alpha$}
\put(-0.1,-0.9){$\beta$}
\end{picture}\ \right\} 
\text{ and }\ 
\setlength{\unitlength}{0.5cm}
p_2=\left\{\ \begin{picture}(0.2,1.5)\thicklines
\put(0.15,-0.25){\line(0,1){0.3}}
\put(0.15,0.35){\line(0,1){0.3}}
\put(-0.1,0.9){$\alpha$}
\put(-0.1,-0.9){$\beta$}
\end{picture}\ \right\}.
$$

Since $\alpha$ and $\beta$ are irreducible, we see that $p_1$ is an admissible diagram if and only if $\alpha\simeq\beta$ and $p_2$ is admissible if and only if $\alpha\simeq\beta\simeq1_{\mathbb{G}}$. 

Therefore, if $\alpha$ is not equivalent to $\beta$: $$\text{dim } \Hom_{\mathscr{\mathbb{G}}\wr_*S_N^+}(r(\alpha),r(\beta))=0.$$

If $\alpha\simeq\beta$ are not the trivial corepresentation $1_{\mathbb{G}}$ then the only intertwiner $r(\alpha)\to r(\beta)$ arises from $p_1$:
$$\text{dim } \Hom_{\mathscr{\mathbb{G}}\wr_*S_N^+}(r(\alpha),r(\alpha))=1.$$

If $\alpha\simeq\beta\simeq1_{\mathbb{G}}$, then the diagram $p_2$ also gives rise to an intertwiner $U_{(1_\mathbb{G}),(1_{\mathbb{G}})}^{p_2,S}$ with $S : 1_\mathbb{G}\to1_{\mathbb{G}}$ the trivial inclusion. 
The independence of $T_{\{|\}}=\id_{\mathbb{C}^N}$ and $T_{{\setlength{\unitlength}{0.2cm}
\left\{\ \begin{picture}(0.2,0.3)\thicklines
\put(0.1,-0.45){\line(0,1){0.5}}
\put(0.1,0.45){\line(0,1){0.5}}
\end{picture}\ \right\}}}$
allows to conclude
$$\text{dim } \Hom_{\mathscr{\mathbb{G}}\wr_*S_N^+}(r(1_{\mathbb{G}}),r(1_{\mathbb{G}}))=2.$$



\end{proof}

\section{The free probability of free wreath product quantum groups}
We provide here some probabilistic consequences of the description of the intertwiner spaces of $\mathbb{G}\wr_{*} S_N^{+}$. In this section we are mainly interested in the non-commutative probability space arising from the Haar state on $C(\mathbb{G}\wr_{*} S_N^{+})$ and the behavior of the coefficients of a corepresentation as random variables in this setting. Since most of the results involve the law of free compound poisson laws, we shall recall its definition. We refer to \cite{nica2006lectures} for an introductory course on non-commutative variables.

\subsection{Laws of characters}\label{char}

\begin{nota} In the sequel $\epsilon=\epsilon(1)\dots\epsilon(r)$ denotes a word in $\lbrace 1,*\rbrace$ and $NC(\epsilon)$ is the set of noncrossing partitions with each endpoint $i$ colored with $\epsilon(i)$. For $p\in NC(\epsilon)$ and $B$ a block of $p$, $\epsilon(B)$ denotes the subword of $\epsilon$ coming from the points in the block $B$ (with the same order as in $p$). 

Let $(A,\phi)$ be a noncommutative probability space, $X$ an element of $A$ with $*-$distribution $\mu_{X}$ depicted by all of its moments 
$$m_{X}(\epsilon)=\phi(X^{\epsilon(1)}\dots X^{\epsilon(r)}).$$
Similarly as in \eqref{cldef3}, the free cumulants of $X$, $\lbrace k_{X}(\epsilon)\rbrace_{\epsilon}$ is the unique collection of complex numbers such that the following moment-cumulant formula holds for all $\epsilon$ :
$$m_{X}(\epsilon)=\sum_{p\in NC(\epsilon)}\prod_{B}k_{X}(\epsilon(B)).$$
The existence and unicity of such a collection is easily proven by recurrence on the length of $\epsilon$ \cite{nica2006lectures}.
\end{nota}
\begin{defi}
The free compound poisson distribution $\mathcal{P}_{\lambda}(\mu_{X})$ with initial law $\mu_{X}$ and parameter $\lambda>0$ is the $\star-$distribution defined by its free cumulants 
\begin{equation}
k_{\mathcal{P}_{\lambda}(\mu_{X})}(\epsilon)=\lambda m_{X}(\epsilon).
\label{defcompoisson}
\end{equation}
\end{defi}
In particular, if $Y$ is a random variable following a free compound poisson distribution with initial law $\mu_{X}$ and parameter $1$, then we have the following moment formula :
$$m_{Y}(\epsilon)=\sum_{p\in NC(\epsilon)}\prod_{B}m_{X}(\epsilon({B})).$$
We refer to \cite{nica2006lectures} for the proof that there exists actually a propability space and a random variable on it with such a distribution.\\
The first result is a direct application of the Corollary \ref{corbas}. We refer to Definition \ref{defralpha} for the definition of the corepresentation $r(\alpha)$.
\begin{prp}
Let $\mathbb{G}$ be a compact quantum group of Kac type, $\alpha\in\Rep(\mathbb{G})$, $n\geq 4$. Then the law of the character $\chi(r(\alpha))$, with respect to the Haar state $h$, is a free compound poisson with initial law $\chi(\alpha)$ and parameter $1$.
\end{prp}
\begin{proof}
Let $\epsilon$ be a word in $\lbrace 1,\star\rbrace$. Then the law of a free compound poisson with initial law $\chi(\alpha)$ and parameter $1$, $\mathcal{P}(\chi(\alpha))$ is described by its free cumulants, with the formula \eqref{defcompoisson}:
$$k_{\mathcal{P}(\chi(\alpha))}(\epsilon(1)\dots\epsilon(r))=m_{\chi(\alpha)}(\epsilon(1)\dots\epsilon(r).$$
With the moment-cumulant formula, this is equivalent to the following expression for the moments of $\mathcal{P}(\chi(\alpha))$:
$$m_{\mathcal{P}(\chi(\alpha))}=\sum_{p\in NC_{\epsilon}}\prod_{B}m_{\chi(\alpha)}(\epsilon(B)).$$
By the Corollary \ref{corbas} we have  
\begin{align*}
h\left(\chi(r(\alpha)^{\epsilon(1)}\dots\chi_{n}(r(\alpha))^{\epsilon(r)}\right)
=&\dim\Hom_{\mathbb{G}\wr_*S_N^+}(1;r(\alpha)^{\epsilon(1)}\otimes\dots\otimes r(\alpha)^{\epsilon(r)})\\
=&\sum_{p\in NC_{\epsilon}}\prod_{B}\dim\Hom_{\mathbb{G}}(1,\alpha(L_B))\\
=&\sum_{p\in NC_{\epsilon}}\prod_{B}m_{\chi(\alpha)}(\epsilon(B)).\\
\end{align*}
The second equality is given by Corollay \ref{corbas}, and the third one by the definition of $\alpha(L_{B})$ and the tensor product structure.
\end{proof}
A consequence of this result is a partial answer to the free product conjecture stated by Banica and Bichon (see \cite{banica2007free}) : for each compact matrix quantum group $(A,v)$ we denote by $\mu(A,v)$ the law of the character of the fundamental representation with respect to the Haar measure. A quantum permutation group is a quantum subgroup of $S_{N}^{+}$ for some $N\geq 0$, in the following sense : that is a compact matrix quantum group $(A,v)$ such that there exists a surjective $C^{*}-$morphism $\Phi:C(S_{N}^{+})\rightarrow A$ sending the elements $u_{ij}$ of $C(S_{N}^{+})$ to $v_{ij}$ (see \cite{banica78quantum} for a survey on the subject).
\begin{crl}
Let $(A,v)$ be a quantum permutation group, and $S_N^{+}=(C(S_{N}^{+},u)$, $n\geq 4$. Then 
$$\mu(A\wr_{*} B,w)=\mu(A,v)\boxtimes\mu(C(S_{N}^{+},u).$$
\end{crl}
\begin{proof}
It is a direct consequence of the last proposition and the fact that in the orthogonal case  the law of a free compound poisson with initial law $\mu$ is the same as the free multiplicative convolution of $\mu$ with the free poisson distribution.
\end{proof}
The conjecture asserts that this formula still holds when replacing $S_N^{+}$ with certain quantum subgroups of $S_N^{+}$. See \cite{banica2007free} for more details.\\

\subsection{Weingarten calculus}\label{Weingarten}
We can also elaborate a Weingarten calculus for a free wreath product. It was mainly developped in the framework of compact quantum groups and permutation quantum groups by Banica and Collins (see \cite{banica2005integration},\cite{banica2007integration}). This tool has mainly two advantages : on one hand it allows us sometimes to get some interesting formulae for the Haar state on the matrix entries of a corepresentation, and on the other hand it yields some asymptotic results on the joint law of a finite set of elements when the dimension of the quantum group goes to infinity. 

Let us first sum up the pattern of this method coming from \cite{banica2005integration}: let $\mathbb{G}=(A,(u_{ij})_{1\leq i,j\leq n})$ be a matrix compact quantum group acting on $V^{\otimes k}=\langle X_{i}\rangle_{1\leq i\leq n}^{\otimes k}$ with the corepresentation $\alpha_{k}$, and $h$ the associated Haar measure. We will assume that $\mathbb{G}$ is orthogonal to simplify the notations, although it could be easily generalized to the general Kac type case : that means that the elements $u_{ij}$ are all self-adjoint in $A$ (see \cite{woronowicz1987compact}). By the property of the Haar state, 
$$(Id\otimes h)\circ \alpha_{k}(X_{i_{1}}\otimes\dots\otimes X_{i_{k}})=P(X_{i_{1}}\otimes\dots\otimes X_{i_{k}}),$$
with $P$ the orthogonal projection of $V^{\otimes k}$ on the invariant subspace of $\alpha_{k}$. On the other hand,
$$(Id\otimes h)\circ \alpha_{k}(X_{i_{1}}\otimes\dots\otimes X_{i_{k}})=\sum h(u_{j_{1}i_{1}}\dots u_{j_{k}i_{k}})(X_{j_{1}}\otimes\dots\otimes X_{j_{k}}) .$$
We get thus the following expression for the Haar state on $u_{j_{1}i_{1}}\dots u_{j_{k}i_{k}}$:
$$ h(u_{j_{1}i_{1}}\dots u_{j_{k}i_{k}})=\langle P(X_{i_{1}}\otimes\dots\otimes X_{i_{k}}),X_{j_{1}}\otimes\dots\otimes X_{j_{k}}\rangle.$$
The right-hand side may be hard to compute. Hopefully the Gram-Schmidt orthogonalisation process yields a nicer expression if we already know a basis of the invariant subspace $S_{k}$ of $\alpha_{k}$. Let $\lbrace S_{k}(i)\rbrace$ be a basis of this subspace, $G_{k}$ being the Gram-Schmidt matrix of $\lbrace S_{k}(i)\rbrace$ defined by $G_{k}(i,j)=\langle S_{k}(i),S_{k}(j)\rangle$ and $W_{k}=G_{k}^{-1}$. A standard computation  yields:
$$ h(u_{j_{1}i_{1}}\dots u_{j_{k}i_{k}})=\sum_{i,j} \langle X_{i_{1}}\otimes\dots\otimes X_{i_{k}},S_{k}(i)\rangle W_{k}(i,j)\langle S_{k}(j),X_{j_{1}}\otimes\dots\otimes X_{j_{k}}\rangle.$$
Of course the matrix $W_{k}(i,j)$ is hard to compute.

Let us see nonetheless what it gives in the case of a free wreath product $(\mathbb{G}\wr_{*}S_N^{+},(w_{ij,kl}))$, with $\mathbb{G}$ an orthogonal matrix quantum group. A basis of $S_{k}$ is given by the vectors $U^{p,S}, p\in NC(k)$, as defined in (\ref{ncint2}). The first task is to compute the matrix $W_{k}(i,j)$. Consider the following map 
\begin{align*}
t_k : (\mathbb{C}^N\otimes V)\otimes\dots\otimes (\mathbb{C}^N\otimes V)&\to (\mathbb{C}^N)^{\otimes k}\otimes V\otimes \dots\otimes V\\
\bigotimes_{i=1}^k(x_i\otimes y_i)&\mapsto\bigotimes_{i=1}^kx_i\otimes\bigotimes_{i=1}^ky_i.
\end{align*} 
$t_{k}$ is unitary and and by definition of $U^{p,S}$,
$$t_{k}(U^{p,S})=T_{p}\otimes S.$$
Recall that $S$ depends implicitly on $p$ through the definition \eqref{ncint2}: the latter is an invariant vector of the $k-$tensor product representation of $\mathbb{G}$ that respects the block structure of $p$. Nevertheless $S$ is independant of $N$ and in particular we have the expression
\begin{align*}
\langle U^{p,S},U^{q,S'}\rangle=&\langle t_{k}(U^{p,S}),t_{k}(U^{q,S'})\rangle\\
=&\langle T_{p},T_{q}\rangle\langle S,S'\rangle=N^{b(p\vee q)}\langle S,S'\rangle.
\end{align*}
\begin{rque}
Easy quantum groups form a particular family of compact quantum groups whose associated intertwiners spaces can be combinatorically described. Namely if $\mathbb{G}$ is an easy quantum group, the invariant subspace of the $k-$tensor-product representation is spanned by the vectors $T_{p}$, as defined in Definition\ref{canocons}, with $p$ belonging to a subcategory of $\mathcal{P}(k)$. See \cite{BS09}, \cite{raum2013full} for more informations on the subject, and \cite{kostler2009noncommutative}, \cite{freslonrepresentation} and \cite{Bra11} for some applications. In this case, the scalar product matrix has a simpler form. Indeed if $\mathbb{G}$ is an easy quantum group of dimension $s$ and with category of partition $\mathcal{C}$, then a direct computation yields for $\alpha\leq p,\beta\leq q$ two partitions in $\mathcal{C}$:
$$\langle U^{p,\alpha},U^{q,\beta}\rangle=N^{b(p\vee q)}s^{b(\alpha\vee\beta)}.$$
The Weingarten formula has also a more combinatorial form since we can write:
$$h(w_{i_{1}j_{1},k_{1}l_{1}}\dots w_{i_{r}j_{r},k_{r}l_{r}})=\sum_{\substack{\alpha\leq \ker(\vec{i}),\beta\leq \ker(\vec{j})\\\alpha\leq p\leq\ker(\vec{k}),\beta\leq q\leq \ker(\vec{l})}}G_{k}^{-1}((p,\alpha),(q,\beta)),$$
where $\ker(\vec{i})$ is the partition whose blocks are the set of indices on which $i$ has the same value.
\end{rque}
The  scalar product matrix $G_{k}=(\langle U^{p,S},U^{q,S'}\rangle)_{(p,S),(q,S')}$ is a block matrix, the blocks $G_{k}^{pq}$ being indexed by $p,q\in NC(k)$. Note that as in \cite{banica2005integration}, one can factorize this matrix as following:
$$G_{k}=\Delta_{nk}^{1/2}\tilde{G}\Delta_{nk}^{1/2},$$
where $\Delta_{nk}$ is the diagonal matrix with diagonal coefficients
$$\Delta_{nk}((p,S),(p,S))=N^{b(p)}$$
and 
$$\tilde{G}_{k}((p,S),(q,S')=N^{b(p\vee q)-\frac{b(p)+b(q)}{2}}\langle S,S'\rangle.$$
Asymptotically with $n$ going to infitiny, $\tilde{G}_{k}=D_{k}(1+o(\frac{1}{\sqrt{n}}))$, $D_{k}$ being the block diagonal matrix 
$$D_{k}((p,S),(q,S'))=\delta_{p,q}\langle S,S'\rangle.$$
Finally we can remark that restricted on the subspace $V_{p_{0}}=Vect( (U_{p_{0},S})_{S})$, the matrix  $(\langle S,S'\rangle)_{S,S'}$ is the tensor product of the Gram-Schmidt matrices  of  $\mathbb{G}$ $G_{\mathbb{G},\vert B_{i}\vert}$, for each block $\vert B_{i}\vert$ of $p_{0}$. If we put all these considerations together, we get that 
$$W_{n}((p,S),(q,S'))=\delta_{p,q}N^{-b(p)}\left(\bigotimes_{B\in p}W_{\mathbb{G}}^{-1}\right)(S,S')(1+o(\frac{1}{\sqrt{n}})).$$
This formula allows to generalize the results in \cite{banica2005integration} to the free wreath product case. Define the following partial trace:
\begin{defi}
Let $0\leq s\leq n$ the partial trace of order $s$ of the matrix $w=(w_{ij,kl})_{1\leq i,j\leq r,1\leq k,l\leq n}$ is 
$$\chi^{w}(s)=\sum_{i=1}^{r}\sum_{k=1}^{s} w_{ii,kk}.$$
\end{defi}
The following result holds for a free wreath product with $S_N^{+}$:
\begin{thm}
Let $\mathbb{G}$ be a matrix compact quantum group of Kac type and dimension $r$, $\chi_{\mathbb{G}}$ the law of the character of its fundamental representation. Let $(W^{n},(w_{ij,kl})_{1\leq i,j\leq r,1\leq k,l\leq n})$ be the matrix quantum group $\mathbb{G}\wr_{*}S_N^{+}$ with its fundamental representation $w$. Then with respect to the haar measure, if $s\sim tn$ for $t\in(0,1]$, $n$ going to infinity,
$$\chi^{w}(s)\rightarrow \mathcal{P}_{t}(\chi_{\mathbb{G}}).$$
\end{thm}
\begin{proof}
A similar computation as in \cite{banica2005integration}, Theorem 5.1 gives 
$$\int (\chi^{w}(s))^{k}=Tr(G_{k,n}^{-1}G_{k,s})$$
and with the asymptotic form of $G_{k,n}$ this gives us:
$$G_{k,n}^{-1}G_{k,s}=\Delta_{nk}^{-1/2}D_{k}^{-1}\Delta_{nk}^{-1/2}(Id+o(\frac{1}{\sqrt{n}}))\Delta_{sk}^{1/2}D_{k}\Delta_{sk}^{1/2}(Id+o(\frac{1}{\sqrt{n}})).$$
Since $D_{k}$ is block diagonal and $\Delta_{nk},\Delta_{sk}$ are diagonal, and equal to the identity on each block, these three matrices commute, and 
$$Tr(G_{k,n}^{-1}G_{k,s})=Tr(\Delta_{s/n,k}(Id+o(\frac{1}{\sqrt{n}})))\rightarrow Tr(\Delta_{t,k}).$$
Since 
$$Tr(\Delta_{t,k})=\sum_{p\in NC(k)} t^{b(p)} \dim V_{p}=\sum_{p\in NC(k)}t^{b(p)} \prod_{B\in p}m_{\vert B\vert}(\chi_{\mathbb{G}}).$$
The latter expression is exactly the $k-$th moment of the law $\mathcal{P}_{t}(\chi_{\mathbb{G}})$.
\end{proof}
\begin{rque}
All the results of this section can be transposed to the classical case. One just has to substitute classical compound poisson laws for free compound poisson laws, and use crossing partitions instead of non-crossing ones.
\end{rque}

\section{The monoidal category of free wreath products by quantum permutation groups}
Let $\mathbb{G}$ be a compact matrix quantum group of Kac type, $N\ge4$. In this section, we prove that $\mathbb{G}\wr_*S_{N}^+$ is monoidally equivalent to a compact quantum group $\mathbb{H}$ with $C(\mathbb{H})\subset C(\mathbb{G}*SU_{q}(2))$ and $q+q^{-1}=\sqrt{N}$, $0<q\le1$. In other words, we shall construct $\widehat{\mathbb{H}}$ as a discrete quantum subgroup of $\widehat{\mathbb{G}*SU_{q}(2)}$.

 We denote by $b:=(b_{ij})_{1\le i,j\le 2}$ the generating matrix of $SU_{q}(2)$. Let $\mathbb{H}=(C(\mathbb{H}),\Delta)$ be the compact matrix quantum group with 
$$C(\mathbb{H}):=C^*-\langle b_{ij}ab_{kl}\ |\ 1\le i,j,k,l\le 2, a\in C(\mathbb{G})\rangle\subset C(\mathbb{G})*C(SU_{q}(2)),$$
\begin{align*}
\Delta(b_{ij}ab_{kl})&=\Delta_{SU_{q}(2)}(b_{ij})\Delta_{\mathbb{G}}(a)\Delta_{SU_{q}(2)}(b_{kl})\\
&=\sum_{r,s}(b_{ir}\otimes b_{rj})\sum (a_{(1)}\otimes a_{(2)})\sum_t(b_{kt}\otimes b_{tl})\\
&=\sum_{r,s,t}b_{ir}a_{(1)}b_{kt}\otimes b_{rj}a_{(2)}b_{tl}\in C(\mathbb{H})\otimes C(\mathbb{H}).
\end{align*}

\begin{nota}
For any $\alpha\in \Irr(\mathbb{G})$, we denote $s(\alpha)=b\otimes\alpha\otimes b$.
\end{nota}

We need to recall some notions on Temperley-Lieb diagrams and fix some notation.

Recall that the intertwiner spaces between tensor powers of $b$ in $SU_{q}(2)$ are given by Temperley-Lieb diagrams as follows: 
$$\Hom_{SU_{q}(2)}(b^{\otimes k},b^{\otimes l})=\text{span}\{T_{D} : (\mathbb{C}^2)^{\otimes k}\to (\mathbb{C}^2)^{\otimes l} : D\in TL(k,l)\},$$
where $TL_q(k,l)$ consists of Temperley-Lieb diagrams (non-crossing pairings) between $k$ upper points and $l$ lower points linked by $(k+l)/2$ strings (this set is empty if $k+l$ is odd). 

We define $TL_q(k,l), k,l\in\N$ as the vector space with basis $TL(k,l)$ and the collection of spaces $TL_q(k,l), k,l\in\N$ forms a rigid monoidal $C^*$ category: when performing a composition $D\circ E$ of two such diagrams, closed loops $\{\text{O}\}$ might appear. They correspond to a multiplication by a factor $q+q^{-1}(=\sqrt{N}$, here) in the final vertical concatenation denoted $DE$. We will denote by $b(D,E)$ the number of closed blocks (closed loops in this case) appearing while performing such operations and we then have
$$D\circ E=N^{b(D,E)/2}DE.$$
The collection of spaces $\Hom_{SU_{q}(2)}(b^{\otimes k},b^{\otimes l}), k,l\in\mathbb{N}$ form a rigid monoidal $C^*$-tensor category $\mathscr{C}SU_{q}(2)$ which is generated by $(T_D)_{D\in \langle\mathcal{D}SU_{q}(2)\rangle}$ with
\begin{equation}\label{genon}
\mathcal{D}SU_{q}(2)=\{\left\{\cap\},\{|\}\right\}
\end{equation}
where $\{\cap\}\in TL(0;2), \{|\}\in TL(1;1)$. The set $\langle\mathcal{D}SU_{q}(2)\rangle$ is composed of all the diagrams obtained by usual composition, tensor product and conjugation of diagrams in $\mathcal{D}SU_{q}(2)$. 

We denote by $\tau$ the non-normalized Markov trace on $TL_q(k,k)$ defined as $$\tau(D)=(q+q^{-1})^{C_D}=(\sqrt{N})^{C_D}$$ where $C_D\in\N$ is  the numbers of ``closed curves" appearing when closing a diagram $D\in TL(k,k)$ by strings on top and bottom as follows:
\begin{center}
\begingroup%
  \makeatletter%
  \providecommand\color[2][]{%
    \errmessage{(Inkscape) Color is used for the text in Inkscape, but the package 'color.sty' is not loaded}%
    \renewcommand\color[2][]{}%
  }%
  \providecommand\transparent[1]{%
    \errmessage{(Inkscape) Transparency is used (non-zero) for the text in Inkscape, but the package 'transparent.sty' is not loaded}%
    \renewcommand\transparent[1]{}%
  }%
  \providecommand\rotatebox[2]{#2}%
  \ifx\svgwidth\undefined%
    \setlength{\unitlength}{114.90936748bp}%
    \ifx\svgscale\undefined%
      \relax%
    \else%
      \setlength{\unitlength}{\unitlength * \real{\svgscale}}%
    \fi%
  \else%
    \setlength{\unitlength}{\svgwidth}%
  \fi%
  \global\let\svgwidth\undefined%
  \global\let\svgscale\undefined%
  \makeatother%
  \begin{picture}(1,1.04367313)%
    \put(0,0){\includegraphics[width=\unitlength]{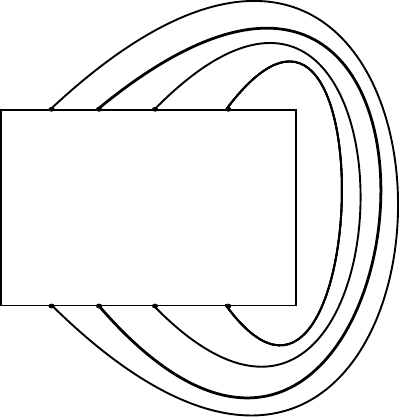}}%
    \put(0.29664702,0.49899459){\color[rgb]{0,0,0}\makebox(0,0)[lb]{\smash{$D$}}}%
  \end{picture}%
\endgroup%

\end{center}
In addition to this pictorial representation, one can define $C_D, D\in TL_q(k,k)$ (and then the Markov trace on $TL_q(k,k)$) by an algebraic formula as follows: consider the conditional expectation $$p_k : TL_q(k+1,k+1)\to TL_q(k,k)$$ obtained by only closing up the last strand:
$$p_k : D\in TL_q(k+1,k+1)\mapsto (id^{\otimes k}\otimes\cup)\circ(D\otimes id)\circ(id^{\otimes k}\otimes\cap)\in TL_q(k,k).$$
Then $$C_D=p_0\circ p_1\circ\dots\circ p_{k-1}(D).$$

We recall that the map $D\mapsto T_D$ is the GNS map associated with $\tau$ on $TL_q$ and that for all diagrams $D,E\in TL(k,k)$, $\text{Tr}(T_D^*T_E)=\tau(D^*E)$. 

We will need the following well known result: the category of representations of $S_{N^2}^+$ is the one of the so-called even part of $O_N^+$. We shall provide a ``diagrammatic" proof of this result based on a result in \cite{chen}: we refer to \cite[Proposition 3.1]{chen} for more details. We denote by $NC_N$ the category of non-crossing partitions in $NC$ with the rule that a closed block corresponds to a factor $N=(\sqrt{N})^2=(q+q^{-1})^2$. As for diagrams in $TL$, we denote by $b(p,p')$ the number of closed blocks appearing when performing a vertical concatenation of composable non-crossing partitions $p,p'$. 

We will also use the notion of ``black region(s)" for $D\in TL(2k,2k)$, denoted $br(D)$. This is defined as follows: we enclose the diagram $D$ in a box, called external box. The lines of the diagram then produce regions in this external box. The first region on the left of the box is shaded white. Then, going away from the left side of the external box, regions having a common line as a boundary are shaded by different colors. One can refer again to \cite{chen}. Let us give an example of diagram $D\in TL(4,4)$ with $br(D)=2$:
\begin{center}
\begingroup%
  \makeatletter%
  \providecommand\color[2][]{%
    \errmessage{(Inkscape) Color is used for the text in Inkscape, but the package 'color.sty' is not loaded}%
    \renewcommand\color[2][]{}%
  }%
  \providecommand\transparent[1]{%
    \errmessage{(Inkscape) Transparency is used (non-zero) for the text in Inkscape, but the package 'transparent.sty' is not loaded}%
    \renewcommand\transparent[1]{}%
  }%
  \providecommand\rotatebox[2]{#2}%
  \ifx\svgwidth\undefined%
    \setlength{\unitlength}{345.65bp}%
    \ifx\svgscale\undefined%
      \relax%
    \else%
      \setlength{\unitlength}{\unitlength * \real{\svgscale}}%
    \fi%
  \else%
    \setlength{\unitlength}{\svgwidth}%
  \fi%
  \global\let\svgwidth\undefined%
  \global\let\svgscale\undefined%
  \makeatother%
  \begin{picture}(1,0.3821785)%
    \put(0,0){\includegraphics[width=\unitlength]{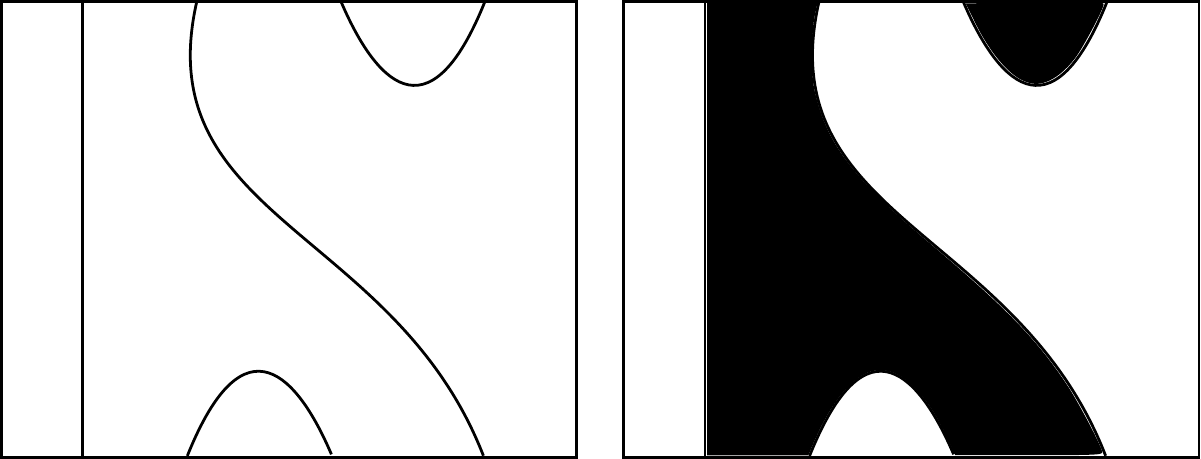}}%
  \end{picture}%
\endgroup%

\end{center}


\begin{prp}\label{insnplus}
Let $N\ge4$ and $0<q\le1$ such that $q+q^{-1}=\sqrt{N}$. Then there exist isomorphisms $$\phi : TL_q(2k,2l)\to NC_N(k,l)$$ for all $k,l\in\mathbb{N}$, all denoted $\phi$ and such that 
\begin{enumerate}
\item $\phi(\id)=\id'$, with $\id=\{|\ |\}\in TL_q(2,2)$ and $\id'=\{|\}\in NC_N(1,1)$,
\item $\phi(D\otimes E)=\phi(D)\otimes\phi(E)$, $\forall D,E\in TL_q$,
\item $\phi(D^*)=\phi(D)^*$, $\forall D\in TL_q$,
\item $\phi(D\circ E)=\phi(D)\circ\phi(E)$, $\forall (D,E)\in TL_q(2l,2m)\times TL_q(2k,2l)$.
\end{enumerate}
\end{prp}

\begin{proof}
 The idea is to use the so-called ``collapsing" operation $c : TL\to NC$ which associates to any Temperley-Lieb diagram, the non-crossing partition obtained by collapsing consecutive neighbors. The converse application is called ``fattening" operation of non-crossing partitions. This latter operation consists in drawing boundary lines around the blocks of $p\in NC(k,l)$ and erasing the original non-crossing partition to obtain a Temperley-Lieb diagram $P\in TL(2k,2l)$, see e.g. \cite{chen}:

\setlength{\unitlength}{0.5cm}
$$
p=\left\{\ \begin{picture}(1.8,2)\thicklines
\put(1,0.3){\line(0,1){1.3}}
\put(0,0.3){\line(1,0){2.028}}
\put(0.029,-1){\line(0,1){1.3}}
\put(2,-1){\line(0,1){1.3}}
\end{picture}\ \ \right\}\mapsto
\left\{\ \begin{picture}(1.8,2)\thicklines
\put(1,0.3){\line(0,1){1.3}}
\put(0.7,0.5){\line(0,1){1.1}}
\put(1.3,0.5){\line(0,1){1.1}}
\put(-0.32,0.5){\line(1,0){1.05}}
\put(1.27,0.5){\line(1,0){1.05}}
\put(0.27,0.1){\line(1,0){1.4}}
\put(0,0.3){\line(1,0){2.028}}
\put(0.029,-1){\line(0,1){1.3}}
\put(-0.3,-1){\line(0,1){1.5}}
\put(2.3,-1){\line(0,1){1.5}}
\put(0.3,-1){\line(0,1){1.1}}
\put(2,-1){\line(0,1){1.3}}
\put(1.68,-1){\line(0,1){1.13}}
\end{picture}\ \ \right\}\mapsto\
P:=\left\{\ \begin{picture}(1.8,2)\thicklines
\put(0.7,0.5){\line(0,1){1.1}}
\put(1.3,0.5){\line(0,1){1.1}}
\put(-0.32,0.5){\line(1,0){1.05}}
\put(1.27,0.5){\line(1,0){1.05}}
\put(0.27,0.1){\line(1,0){1.4}}
\put(-0.3,-1){\line(0,1){1.5}}
\put(2.3,-1){\line(0,1){1.5}}
\put(0.3,-1){\line(0,1){1.1}}
\put(1.68,-1){\line(0,1){1.13}}
\end{picture}\ \ \right\}.
$$
It is proved in \cite[Theorem 4.2]{chen} that the collapsing operation $D\mapsto c(D)$ provides, for all $k$, an isomorphism
$$\psi : D\in TL_q(2k,2k)\mapsto \frac{\tau(D)}{\widetilde{\tau}(c(D))}c(D)\in NC_N(k,k)$$ with $\tau$, $\widetilde{\tau}$ are non-normalized Markov traces on $TL_q$, $NC_N$. In particular following \cite{chen}, we have 
\begin{align}\label{prptra}
&\tau(\id)=(q+q^{-1})^2=N=\widetilde{\tau}(\id') \text{ i.e. } \phi(\id)=\id'.
\end{align}
It is also proved in \cite{chen} that the map $\psi$ satisfies all four relations $(1),(2),(3),(4)$ in the statement. To prove this, they use in particular an alternative definition of $\psi$ : $D\in TL_q(2k,2k), \psi(D)=\sqrt{N}^{k-br(D)}c(D)$ where $br(D)$ is the number of black regions in $D$.

The notion of black region for diagrams $D\in TL_q(2k,2l)$ can be defined the same way as in the case $k=l$. In particular, notice that if $D\in TL(2k,2l)$ then for all $m\in \mathbb{N}$:
\begin{align}\label{capcup}
br(D\otimes \cap^{\otimes m})=br(D)+m=br(D\otimes\cup^{\otimes m}).
\end{align}

We now define $\phi : TL_q(2k,2l)\to NC_N(k,l)$ for all $k,l$ by 
$$D\mapsto N^{\frac{k+l}{4}-\frac{br(D)}{2}}c(D).$$

One can easily see that $\psi$ satisfies relations $(1)$ and $(3)$. For $(2)$ and $(4)$ we use the fact that $\phi|_{TL_q(2k,2k)}=\varphi$:

$(2):$ Let $D\in TL_q(2k,2l)$ and $E\in TL_q(2k',2l')$. Consider $D'\in TL_q(2\max(k,l),2\max(k,l))$ defined by :
\begin{enumerate}
\item[$\bullet$] $D'=D\otimes \cap^{k-l}\in TL_q(2k,2k)$ if $k\ge l$,
\item[$\bullet$] $D'=D\otimes \cup^{l-k}\in TL_q(2l,2l)$ if $l\ge k$,
\end{enumerate}
and the analogue construction for $E$. Then from the case $k+k'=l+l'$, we get :
$$\phi(D'\otimes E')=\phi(D')\otimes\phi(E'),$$
with 
$$D'\otimes E'\in TL_q(2(\max(k,l)+\max(k',l')),2(\max(k,l)+\max(k',l'))).$$
We have:
\begin{enumerate}
\item[$\bullet$] $\phi(D'\otimes E')=N^{\frac{\max(k,l)+\max(k',l')}{4}-\frac{br(D'\otimes E')}{2}}c(D')\otimes c(E')$.  The exponent of $N$ appearing in this expression is then :
\begin{align*}
&\frac{\max(k,l)}{4}+\frac{\max(k',l')}{4}-\frac{br(D'\otimes E')}{2}\\
&\overset{(\ref{capcup})}{=}\frac{\max(k,l)}{4}+\frac{\max(k',l')}{4}-\frac{br(D\otimes E)}{2}-\frac{|k-l|+|k'-l'|}{2}\\
&=\frac{\min(k,l)+\min(k',l')}{4}-\frac{br(D\otimes E)}{2}.\\
\end{align*}
\item[$\bullet$] $\phi(D')\otimes \phi(E')=N^{\frac{\max(k,l)}{4}-\frac{br(D')}{2}}c(D')\otimes N^{\frac{\max(k',l')}{4}-\frac{br(E')}{2}}c(E')$. The exponent of $N$ appearing in this expression is then 
\begin{align*}
&\frac{\max(k,l)}{4}-\frac{br(D')}{2}+\frac{\max(k',l')}{4}-\frac{br(E')}{2}\\
&\overset{(\ref{capcup})}{=}\frac{\max(k,l)}{4}-\frac{br(D)}{2}-\frac{|k-l]}{2}+\frac{\max(k',l')}{4}-\frac{br(E)}{2}-\frac{|k'-l']}{2}\\
&=\frac{\min(k,l)+\min(k',l')}{4}-\frac{br(D)+br(E)}{2}.\\
\end{align*}
\end{enumerate}
We deduce from this, that $br(D\otimes E)=br(D)+br(E)$ and thus $\phi(D\otimes E)=\phi(D)\otimes\phi(E)$.

$(4):$ Consider $D\in TL_q(2l,2m)$ and $E\in TL_q(2k,2l)$. There is nothing to do in the case $k=l=m$ since then $\phi=\psi$.

Now, we first suppose that $l>k,m$. In this case, consider $$E'=E\otimes \{\cup\}^{\otimes {l-k}},\ \ D'=D\otimes \cap^{\otimes {l-m}}.$$ From the case $k=l=m$, we get $$\phi(D'\circ E')=\phi(D')\circ\phi(E').$$ Hence by compatibility of $\phi$ with tensor product of diagrams, we have:
\begin{align*}
\phi(D\circ E)\otimes\phi(\cap)^{\otimes {l-m}}\otimes\phi(\cup)^{\otimes {l-k}}=\left(\phi(D)\circ\phi(E)\right)\otimes\phi(\cap)^{\otimes {l-m}}\otimes\phi(\cup)^{\otimes {l-k}}
\end{align*}
and this implies $\phi(D\circ E)=\phi(D)\circ\phi(E)$.

Now, we suppose that $m\le l<k$. We consider $$E'=E\otimes \{\cap\}^{\otimes {k-l}},\ \ D'=D\otimes \{\cap\}^{{l-m}}\otimes \{|\ |\}^{\otimes {k-l}}.$$ Then again, $\phi(D'\circ E')=\phi(D')\circ\phi(E')$ and we get
\begin{align*}
\phi(D\circ E)\otimes\phi(\cap)^{\otimes {k-m}}=\left(\phi(D)\circ\phi(E)\right)\otimes\phi(\cap)^{\otimes {k-m}}
\end{align*}
which implies $\phi(D\circ E)=\phi(D)\circ\phi(E)$. One can proceed similarly if $k\le l<m$.

To conclude, consider the case $l<k\le m$ (or the analogue $l<m\le k$). We get the desired relation considering $$D'=D\otimes \{\cup\}^{\otimes {m-l}},\ \ E'=E\otimes \{\cap\}^{\otimes {k-l}}\otimes \{|\ |\}^{\otimes {m-k}}.$$ Indeed,
\begin{align*}
\phi(D'\circ E')&=\phi(D\circ E)\otimes \phi(\cup\circ\cap)^{\otimes {k-l}}\otimes \phi(\cup)^{\otimes {m-k}},\\
\phi(D')\circ\phi(E')&=(\phi(D)\circ\phi(E))\otimes (\phi(\cup)\circ\phi(\cap))^{\otimes {k-l}}\otimes\left(\underbrace{\phi(\cup)\circ\phi(|\ |))^{\otimes m-k}}_{\phi(\cup)^{\otimes m-k}}\right)
\end{align*}
and
\begin{enumerate}
\item[$\bullet$] $\phi(\cup\circ\cap)=\phi(\sqrt{N})=\sqrt{N}$,
\item[$\bullet$] $\phi(\cup)\circ\phi(\cap)=N^{\frac{1}{4}-\frac{1}{2}}c(\cup)\circ N^{\frac{1}{4}-\frac{1}{2}}c(\cap)=\frac{1}{N^{\frac{1}{4}}N^{\frac{1}{4}}}N=\sqrt{N}$ by the facts that $c(\cup)\circ c(\cap)=N^{b(\cup),b(\cap)}$ and $b(c(\cup),c(\cap))=1$.
\end{enumerate}  
\end{proof}

We obtain as a corollary that the map $\phi : TL_q(2k,2l)\to NC_N(k,l)$ is isometric. We denote $(\cdot|\cdot)_q$ (respectively $(\cdot|\cdot)_N$) the scalar product on $TL_q$ (respectively $NC_N$) implemented by the non-normalized Markov trace $\tau$ (respectively $\widetilde{\tau}$).
\begin{crl}\label{crlisom}
For all $D,E\in TL_q(2k,2l)$, we have $(D|E)_q=(\phi(D)|\phi(E))_N$.
\end{crl}
\begin{proof}
The result follows from the following computation:
\begin{align*}
\widetilde{\tau}(\phi(D)^*\circ\phi(E))&=\widetilde{\tau}(\phi(D^*\circ E))\\
&=\widetilde{\tau}\left(\frac{\tau(D^*\circ E)}{\widetilde{\tau}(c(D^*\circ E))}c(D^*\circ E)\right)\\
&=\tau(D^*\circ E).
\end{align*}
\end{proof}

We shall apply \cite[Proposition 2.15]{Lem13b} to the special case of the free product $\mathbb{G}*SU_q(2)$, for any compact matrix quantum groups of Kac type $\mathbb{G}$ and to the objects $s(\alpha)=b\otimes\alpha\otimes b$, $\alpha\in \Irr(\mathbb{G})$. We need to introduce several notation.

\begin{nota}\label{Pp}
Let $k,l\in \mathbb{N}$ and $\alpha_1,\dots,\alpha_k,\beta_1,\dots,\beta_l\in \Irr(\mathbb{G})$. We put $[\alpha]=\{b,\alpha_1,b,\dots,b,\alpha_k,b\}$, $[\beta]=\{b,\beta_1,b,\dots,b,\beta_l,b\}$. We denote by $$NC_{\mathbb{G}*SU_{q}(2)}([\alpha],[\beta])\subset NC(3k,3l)$$ the set of non-crossing partitions $D\in NC(3k,3l)$ of $\{1,\dots,3(k+l)\}$ such that one can write $D=P\cup p$ as a disjoint union of partitions where:
\begin{enumerate}
\item[$\bullet$] $P$ is a non-crossing partition on the points $i\in\{1,\dots,3(k+l)\}$ with $i\equiv 0,2 \mod 3$. The points of $P$ are decorated by the fundamental corepresentation $b$ of $SU_{q}(2)$. We denote $T_P$ the natural linear map arising from $P$ as in Definition \ref{canocons}.
\item[$\bullet$] $p$ is a non-crossing partition on the points $i\in\{1,\dots,3(k+l)\}$ with $i\equiv 1 \mod 3$. The points of $p$ are decorated by the $\mathbb{G}$-representations $\alpha_i,\beta_j$ on top and bottom,
\item[$\bullet$] The blocks $B\in p$, $B=U_{B}\cup L_{B}$ of $p$ are decorated by $\mathbb{G}$-morphisms $S_B$.
\end{enumerate}
We denote $S=\bigotimes_{B}S_{B}$ with the order on the blocks of $p$ as in Notation \ref{ordb} and $$V^{D,S}=(\sigma_D^L)^*\circ (T_P\otimes S)\circ \sigma_D^U : \bigotimes_{i=1}^k\mathbb{C}^2\otimes H^{\alpha_i}\otimes\mathbb{C}^2\to\bigotimes_{j=1}^l\mathbb{C}^2\otimes H^{\beta_{j}}\otimes\mathbb{C}^2,$$
where $\sigma_D^U=(\id\otimes\rho_D^U)\lambda_D^U$ and $\sigma_D^L=(\id\otimes\rho_D^L)\lambda_D^L$ with:
\begin{enumerate}
\item[$\bullet$] $\lambda_D^U : \bigotimes_{i=1}^k\mathbb{C}^2\otimes H^{\alpha_{i}}\otimes\mathbb{C}^2\to (\mathbb{C}^2)^{\otimes 2k}\otimes\bigotimes_{i=1}^kH^{\alpha_i}$, $\bigotimes_i(x_i\otimes y_i \otimes x'_i)\mapsto\bigotimes_i(x_i\otimes x_i')\otimes\bigotimes_{i'}y_{i'}$,
\item[$\bullet$] $\rho_D^U : \bigotimes_{i=1}^kH^{\alpha_i}\to \bigotimes_{B\in p} H^{U_{B}}$, $\bigotimes_iy_i\mapsto\bigotimes_{B\in p}\bigotimes_{i\in U_{B}}y_i$,
\item[$\bullet$] $\lambda_D^L : \bigotimes_{j=1}^l\mathbb{C}^2\otimes H^{\beta_{j}}\otimes\mathbb{C}^2\to (\mathbb{C}^2)^{\otimes 2l}\otimes\bigotimes_{j=1}^lH^{\alpha_j}$, $\bigotimes_j(x_j\otimes y_j \otimes x'_j)\mapsto\bigotimes_j(x_j\otimes x_j')\otimes\bigotimes_{j'}y_{j'}$,
\item[$\bullet$] $\rho_D^L : \bigotimes_{j=1}^lH^{\beta_i}\to \bigotimes_{B\in p} H^{L_{B}}$, $\bigotimes_jy_j\mapsto\bigotimes_{B\in p}\bigotimes_{j\in L_{B}}y_j$
\end{enumerate}
\end{nota}


\begin{lem}\label{homfreepr}
Let $\alpha_1,\dots,\alpha_k,\beta_1,\dots,\beta_l\in \Irr(\mathbb{G})$. Then with the above notation
\begin{align}
\label{ff}\Hom_{\mathbb{H}}&\left(\bigotimes_{i=1}^ks(\alpha_i);\bigotimes_{j=1}^ls(\beta_j)\right)\\
&=\emph{span}\left\{V^{D,S} : D\in NC_{\mathbb{G}*SU_{q}(2)}([\alpha],[\beta]),\ S \emph{ as in Notation \ref{Pp}}\right\}\label{ff1}.
\end{align}
\end{lem}
\begin{proof}
The Lemma follows from \cite[Proposition 2.15]{Lem13b} in the case of the free product of two compact matrix groups and form the description of the generating set $\mathcal{D}SU_{q}(2)$, we recalled above (\ref{genon}).

By \cite[Proposition 2.15]{Lem13b}, intertwiners in $\mathbb{G}*SU_q(2)$ are linear combinations of compositions and tensor products of maps $\id\otimes R\otimes id$ where $R$ is:
\begin{enumerate}
\item[$\bullet$] either an intertwiner $R : \alpha_{i_1}\otimes\dots\otimes\alpha_{i_m}\to\beta_{j_1}\otimes\dots\otimes\beta_{j_n}$ in $\mathbb{G}$
\item[$\bullet$] or $\id : b\to b$, $\cap : 1\to b\otimes b$, $\cup : b\otimes b\to 1$.
\end{enumerate}
The inclusion of the space (\ref{ff}) in the space (\ref{ff1}) then follows since we restrict to intertwiners of the type: $(b\otimes\alpha_1\otimes b)\otimes\dots\otimes(b\otimes\alpha_k\otimes b)\to(b\otimes\beta_1\otimes b)\otimes\dots\otimes(b\otimes\beta_l\otimes b)$.


The inclusion $\supset$ holds because any diagram in $NC_{\mathbb{G}*SU_q(2)}$ decomposes as a vertical concatenation of diagrams $\{|\}$, $\{\cap\}$, $\{\cup\}$ and non-crossing partitions $p$ whose points are decorated by $\mathbb{G}$-representations and blocks by $\mathbb{G}$-morphisms. 
\end{proof}

Thanks to the previous result, we shall construct isomorphisms between certain $\mathbb{G}\wr_*S_{N}^+$-Hom spaces and $\mathbb{H}$-Hom spaces and proceed towards proving the monoidal equivalence result we announced in the beginning of this section. 

From a non-crossing partition $D\in NC_{\mathbb{G}*SU_{q}(2)}([\alpha],[\beta])$, we can construct a new non-crossing partition $\widetilde{D}$ decorated by $\mathbb{G}$-representations and by $\mathbb{G}$-morphisms on its blocks. These diagrams will allow us to construct the monoidal equivalence we mentioned in the introduction of this section. We use the notations of Section \ref{intfreec} and Notation \ref{Pp}. In particular, we use the twisting isomorphisms $$t_p^{U}\in\mathcal{B}\left(\bigotimes_{i=1}^k(\mathbb{C}^N\otimes H^{\alpha_i}), (\mathbb{C}^N)^{\otimes k}\otimes\bigotimes_{B(p)} H^{U_{B(p)}}\right)$$ and $$t_p^{L}\in\mathcal{B}\left(\bigotimes_{j=1}^l(\mathbb{C}^N\otimes H^{\beta_j}), (\mathbb{C}^N)^{\otimes l}\otimes\bigotimes_{B(p)} H^{L_{B(p)}}\right)$$ defined in Notation \ref{cantwist}. They can be decomposed as follows:
\begin{nota}
We put $t_p^U=(\id\otimes g_p^U)f_p^U$, $t_p^L=(\id\otimes g_p^L)f_p^L$ with 
\begin{enumerate}
\item[$\bullet$] $f_p^U : \bigotimes_{i=1}^k\mathbb{C}^N\otimes H^{\alpha_{i}}\to (\mathbb{C}^N)^{\otimes k}\otimes\bigotimes_{i=1}^kH^{\alpha_i}$, $\bigotimes_i(x_i\otimes y_i)\mapsto\bigotimes_ix_i\otimes\bigotimes_{i'}y_{i'}$,
\item[$\bullet$] $g_p^U : \bigotimes_{i=1}^kH^{\alpha_i}\to \bigotimes_{p\in B} H^{U_{B}}$, $\bigotimes_iy_i\mapsto\bigotimes_{p\in B}\bigotimes_{i\in U_{B}}y_i$,
\item[$\bullet$] $f_p^L : \bigotimes_{j=1}^l\mathbb{C}^N\otimes H^{\beta_{j}}\to (\mathbb{C}^N)^{\otimes l}\otimes\bigotimes_{j=1}^lH^{\beta_j}$, $\bigotimes_j(x_j\otimes y_j)\mapsto\bigotimes_jx_j\otimes\bigotimes_{j'}y_{j'}$,
\item[$\bullet$] $g_p^L : \bigotimes_{j=1}^lH^{\beta_j}\to \bigotimes_{p\in B} H^{L_{B}}$, $\bigotimes_jy_j\mapsto\bigotimes_{p\in B}\bigotimes_{j\in U_{B}}y_j$.
\end{enumerate}
\end{nota}

\begin{rque}\label{egrho}
Let $D\in NC_{\mathbb{G}*SU_{q}(2)}([\alpha],[\beta]),\ D=P\cup p$. From the definitions, it follows immediately that $g_p^L=\rho_D^L$ and $g_p^U=\rho_D^U$.
\end{rque}

\begin{defi} Let $k,l\in \mathbb{N}$ and $\alpha_1,\dots,\alpha_k,\beta_1,\dots,\beta_l\in \Irr(\mathbb{G})$. Let $$D\in NC_{\mathbb{G}*SU_{q}(2)}([\alpha],[\beta]),\ D=P\cup p$$ with a decoration of the blocks by a morphism $S$ as above in Notation \ref{Pp}.

We define $\widetilde{D}$ as follows: we identify each point decorated by a $\mathbb{G}$-representation with the $2$ adjacent points decorated by the fundamental representation $b$ of $SU_{q}(2)$ and we keep the same decoration by the $\mathbb{G}$-representations. $\widetilde{D}$ is then the quotient non-crossing partition generated by $D$. In other words, $\widetilde{D}$ is the coarsest non-crossing partition such the image $\bar B$ under the identification above of any block $B\in D$, $B\subset\{1,\dots,3(k+l)\}$, is included in a block $\widetilde{B}$ of $\widetilde{D}$. Notice that with our notation, $\widetilde{D}=c(P)$.
 
A block $\widetilde{B}$ of $\widetilde{D}$ is decorated by the tensor product of the maps $S_B$ decorating the blocks $B\in p$ such that $\bar B\subset \widetilde{B}$. The resulting map $\widetilde{S}$ is:
$$\widetilde{S}=g_{\widetilde{D}}^L\circ (g_p^L)^*\circ\bigotimes_{B(p)}S_{B(p)}\circ g_p^U\circ(g_{\widetilde{D}}^U)^*$$

\end{defi}
Now, we can prove:

\begin{prp}\label{intsim}
Let $N\ge4$ and $k,l\in \mathbb{N}$. For all $\alpha_1,\dots,\alpha_k,\beta_1,\dots,\beta_l\in \Irr(\mathbb{G})$, there exists a linear isomorphism
\begin{align*}
\varphi : \Hom_{\mathbb{H}}&(s(\alpha_1)\otimes\dots\otimes s(\alpha_k) ; s(\beta_1)\otimes\dots\otimes s(\beta_l))\\
&\overset{\sim}{\longrightarrow} \Hom_{\mathbb{G}\wr_*S_{N}^+}(r(\alpha_1)\otimes\dots\otimes r(\alpha_k); r(\beta_1)\otimes\dots\otimes r(\beta_l))
\end{align*}
such that 
\begin{enumerate}
\item[$\bullet$] $\varphi(\id)=\id$,
\item[$\bullet$] $\varphi(Q\otimes R)=\varphi(Q)\otimes\varphi(R)$ for all $\mathbb{H}$-morphisms $Q,R$,
\item[$\bullet$] $\varphi(Q^*)=\varphi(Q)^*$ for all $\mathbb{H}$-morphisms $Q$,
\item[$\bullet$] $\varphi(Q\circ R)=\varphi(Q)\circ\varphi(R)$ for all composable $\mathbb{H}$-morphisms $Q,R$.
\end{enumerate}
\end{prp}
\begin{proof}

With the above notation and the one of Theorem \ref{bigthm}, we define $\varphi$ by $\varphi\left(V^{D,S}\right)=U^{\widetilde{D},\widetilde{S}}.$
The fact that this map is well defined will follow from the following lemma:
\begin{lem} For all $D=P\cup p, D'=P\cup p'\in NC_{\mathbb{G}*SU_q(2)}([\alpha],[\beta])$ and any decoration $S,S'$ of the blocks of $D,D'$ we have
$$\Tr\left(\left(V^{D,S}\right)^*V^{D',S'}\right)={\Tr}\left(\left(U^{\widetilde{D},\widetilde{S}}\right)^*U^{\widetilde{D'},\widetilde{S'}}\right).$$
\end{lem}
\begin{proof}
We first compute
\begin{align*}
&\Tr\left((V^{D,S})^*V^{D',S'}\right)=\Tr\left((\sigma_D^U)^* (T_P^*\otimes S^*) \sigma_D^L (\sigma_{D'}^L)^* (T_{P'}\otimes S') \sigma_{D'}^U\right)\\
&=\Tr\left((\lambda_D^U)^*(\id\otimes\rho_D^U)^* (T_P^*\otimes S^*) (\id\otimes\rho_D^L)\lambda_D^L (\lambda_{D'}^L)^*(\id\otimes\rho_{D'}^L)^* (T_{P'}\otimes S') (\id\otimes\rho_{D'}^U)\lambda_{D'}^U\right).\\
\end{align*}
Using the facts that $\lambda_{D'}^U(\lambda_{D}^U)^*=\id$, $\lambda_{D}^L(\lambda_{D'}^L)^*=\id$ and that $\Tr$ is a trace, we obtain
\begin{align*}
&\Tr\left((V^{D,S})^*V^{D',S'}\right)=\Tr\left((\id\otimes\rho_D^U)^* (T_P^*\otimes S^*) (\id\otimes\rho_{D}^L)(\id\otimes\rho_{D'}^L)^* (T_{P'}\otimes S') (\id\otimes\rho_{D'}^U)\right)\\
&=\Tr\left((T_P^*T_{P'})\otimes (\rho_D^U)^*S^*\rho_D^L(\rho_{D'}^L)^*S'\rho_{D'}^U\right)=\Tr\left(T_P^*T_{P'}\right)\Tr\left((\rho_D^U)^*S^*\rho_D^L(\rho_{D'}^L)^*S'\rho_{D'}^U\right).
\end{align*}
We get similarly
\begin{align*}
{\Tr}\left(\left(U^{\widetilde{D},\widetilde{S}}\right)^*U^{\widetilde{D'},\widetilde{S'}}\right)&={\Tr}\left(T_{\widetilde{D}}^*T_{\widetilde{D}'}\right){\Tr}\left((g_{\widetilde{D}}^U)^*\widetilde{S}^*g_{\widetilde{D}}^L(g_{\widetilde{D}'}^L)^*\widetilde{S}'g_{\widetilde{D}'}^U\right).
\end{align*}
Now, since $\widetilde{D}=c(P)$, we obtain $$\Tr\left(T_P^*T_{P'}\right)=\tau(P^*P')=\widetilde{\tau}(\widetilde{D}^*\widetilde{D}')={\Tr}\left(T_{\widetilde{D}}^*T_{\widetilde{D}'}\right)$$ by Corollary \ref{crlisom}. On the other hand, we have
\begin{align*}
(\rho_D^U)^*S^*\rho_D^L(\rho_{D'}^L)^*S'\rho_{D'}^U&=(\rho_D^U)^*\circ\bigotimes_{B(p)}S_{B(p)}^*\circ\rho_D^L(\rho_{D'}^L)^*\circ\bigotimes_{B(p')}S'_{B(p')}\circ\rho_{D'}^U
\end{align*}
and
\begin{align*}
&(g_{\widetilde{D}}^U)^*\widetilde{S}^*g_{\widetilde{D}}^L(g_{\widetilde{D}'}^L)^*\widetilde{S}'g_{\widetilde{D}'}^U\\
&=(g_{\widetilde{D}}^U)^*\left(g_{\widetilde{D}}^L\circ (g_p^L)^*\circ\bigotimes_{B(p)}S_{B(p)}\circ g_p^U\circ(g_{\widetilde{D}}^U)^*\right)^*g_{\widetilde{D}}^L(g_{\widetilde{D}'}^L)^*\left(g_{\widetilde{D}'}^L(g_{p'}^L)^*\circ\bigotimes_{B(p')}S'_{B(p')}\circ g_{p'}^U(g_{\widetilde{D}'}^U)^*\right)g_{\widetilde{D}'}^U\\
&=(g_p^U)^*\circ\bigotimes_{B(p)}S_{B(p)}^*\circ g_p^L(g_{p'}^L)^*\circ\bigotimes_{B(p')}S'_{B(p')}\circ g_{p'}^U.
\end{align*}
The lemma follows from these calculations and Remark \ref{egrho}.
\end{proof}
Now, since the trace on $TL_q$ (and $NC_N$) is faithful, we deduce from the above lemma that the map $\varphi$ map is well defined. Indeed, the map $(D,S)\mapsto V^{D,S}$ factorizes through the quotient map $(D,S)\mapsto (\widetilde{D},\widetilde{S})$, for all $D\in NC_{\mathbb{G}*SU_q(2)}([\alpha],[\beta])$ and all decorations $[S]$ of the blocks of $D$ since it is isometric. Notice that it proves also that $\varphi$ is injective.



Let $U^{p,S}$ be an element of the basis of the vector space $\Hom_{\mathbb{G}\wr_*S_{N}^+}(r(\alpha_1)\otimes\dots\otimes r(\alpha_k); r(\beta_1)\otimes\dots\otimes r(\beta_l))$ obtained in Theorem \ref{bigthm}. Consider $P$ the Temperley-Lieb diagram obtained by the fattening operation on $p$ and put $D=P\cup p$. Notice that in this case, we have $\widetilde{D}=p$, $\widetilde{S}=S$. It is then clear that $\varphi(V^{D,S})=U^{\widetilde{D},\widetilde{S}}=U^{\widetilde{D},S}$ and the $\varphi$ is surjective.

In particular, we have proved that the morphisms $V^{D,S}\in \Hom_{\mathbb{H}}$ such that $\widetilde{D}=p$ generate a basis of the Hom-spaces in $\mathbb{H}$.


Now notice at the level of diagrams describing the categories of morphisms, that $\varphi$ satisfies all first three relations of the statement. The relation $\varphi(D\circ E)=\varphi(D)\circ\varphi(E)$ for all $D,E\in D\in NC_{\mathbb{G}*SU_{q}(2)}$ follows from Proposition \ref{insnplus}. 

Indeed, let $D=P\cup p$ and $E=P'\cup p'$ as in Notation \ref{Pp}. We may assume that $\widetilde{D}=p$ and $\widetilde{E}=p'$. When we compose diagrams $D=P\cup p$ and $E=P'\cup p'$, we compose on the one hand Temperley-Lieb diagrams $P\circ P'$ and on the other hand non-crossing partitions $p\circ p'$. Hence closed loops in $P\circ P'$ and closed blocks in $p\circ p'$ might appear:
\begin{enumerate}
\item[$\bullet$] We know with notation of Proposition \ref{insnplus} that $\phi(P\circ P')=\phi(P)\circ \phi(P')$,
\item[$\bullet$] Closed blocks in $p\circ p'\subset\varphi(D\circ E)$ correspond to scalars coefficients $$\mathbb{C}\to \alpha_t^p\otimes\dots\otimes\alpha_{t+m}^p\to\beta_{r}^{p'}\otimes\dots\otimes\beta_{r+n}^{p'}\to \mathbb{C},$$ for some $\mathbb{G}$-representations $\alpha_t^p,\dots,\alpha_{t+m}^p\in B$, ($B\in p$) and $\beta_r^{p'},\dots,\beta_{r+n}^{p'}\in B'$, ($B'\in p'$). These scalars also appear precisely in $\widetilde{D}\circ\widetilde{E}$.
\end{enumerate}
Altogether, $$\varphi(D\circ E)=\varphi(D)\circ\varphi(E).$$

\end{proof}

We can now prove the main result of this section.

\begin{thm}\label{resmono}
Let $N\ge4$ and $0<q\le 1$ such that $q+q^{-1}=\sqrt{N}$. Let $\mathbb{G}$ be a compact matrix quantum group of Kac type. Then $$\mathbb{G}\wr_*S_N^+\simeq_{mon} \mathbb{H}$$ where $\widehat{\mathbb{H}}$ is the subgroup of $\widehat{\mathbb{G}}*\widehat{SU_{q}(2)}$ with 
$$C(\mathbb{H}):=C^*-\langle b_{ij}ab_{kl}\ |\ 1\le i,j,k,l\le 2, a\in C(\mathbb{G})\rangle\subset C(\mathbb{G})*C(SU_{q}(2)),$$
\begin{align*}
\Delta(b_{ij}ab_{kl})=\sum_{r,s,t}b_{ir}ab_{kt}\otimes b_{rj}ab_{tl}\in C(\mathbb{H})\otimes C(\mathbb{H})
\end{align*}
and $B=(b_{ij})_{ij}$ is the generating matrix of $SU_{q}(2)$.
\end{thm}
\begin{proof}
With the notation of this section and the one of Section \ref{intfreec}, the representation categories of $\mathbb{H}$ and $\mathbb{G}\wr_*S_N^+$ are respectively given by the completions (in the sense of Woronowicz \cite{Wor88}) of $$R_0=\left\{s(\alpha_1)\otimes\dots\otimes s(\alpha_k) :     \begin{array}{lll}
        k\in\N\\
        \alpha_i\in\Irr(\mathbb{G})\\
        i=1,\dots,k
    \end{array}\right\} , 
    \left\{V^{D,S} : \begin{array}{llll}
       k,l\in\N\\    
        \alpha_i,\beta_j\in\Irr(\mathbb{G})\\
        D\in NC_{\mathbb{G}*SU_q(2)}([\alpha],[\beta])\\
        S \text{ decorates the blocks of } D
    \end{array}\right\}$$

and 

$$S_0=\left\{r(\alpha_1)\otimes\dots\otimes r(\alpha_k) :     \begin{array}{ll}
        k\in\N\\
        \alpha_i\in\Irr(\mathbb{G})
    \end{array}\right\} , 
    \left\{U^{p,S} : \begin{array}{llll}
       k,l\in\N\\    
        \alpha_i,\beta_j\in\Irr(\mathbb{G})\\
        p\in NC_{\mathbb{G}}([\alpha],[\beta])\\
        S \text{ decorates the blocks of } p
    \end{array}\right\}.$$

We have already constructed in Proposition \ref{intsim} an equivalence $\varphi$ between the monoidal rigid $C^*$-tensor categories $R_0,S_0$. Thanks to Lemma \ref{extcomp}, one can extend it to the completions $R,S$.

To conclude, notice that $\varphi : R_0\to S_0$ is an equivalence of categories in such a way that the generators $s(\alpha)$ of $C(\mathbb{H})$ and $r(\alpha)$ of $C(\mathbb{G}\wr_*S_N^+)$ are in correspondence. We deduce easily that $\varphi$ induces a bijection $$\Irr(\mathbb{H})\to\Irr(\mathbb{G}\wr_*S_N^+),\ \omega'\mapsto\omega$$ where $\omega$ is constructed as in the proof of Lemma \ref{extcomp} by
$$\omega={\varphi}(v){\varphi}(v^*)r(\alpha_1)\otimes\dots\otimes r(\alpha_k)$$
for all $\text{Irr}(\mathbb{H})\ni \omega'\subset\alpha_1\otimes\dots\otimes\alpha_k$ and with $v$ the isometry $$v : \omega'\to s(\alpha_1)\otimes\dots\otimes s(\alpha_k).$$
The fact that $\varphi$ is a monoidal equivalence between $\mathbb{H}$ and $\mathbb{G}\wr_*S_N^+$ then follows.
\end{proof}

We get as an immediate corollary the description of the fusion rules of $\mathbb{G}\wr_*S_N^+$. Even if this result can be formulated as Proposition \ref{fusdansprod}, we give a closer formulation as \cite[Theorem 7.3]{BV09} and \cite[Theorem 2.25]{Lem13b}.
\begin{defi}
Let $M=\langle \Irr(\mathbb{G})\rangle$ be the monoid formed by the words over $\Irr(\mathbb{G})$. We endow $M$ with the following operations:
\begin{enumerate}
\item Involution: $\overline{(\alpha_{1},\dots ,\alpha_{k})}=(\bar\alpha_{k},\dots, \bar\alpha_{1})$,
\item concatenation: for any two words, we set $$(\alpha_{1},\dots,\alpha_{k}),(\beta_{1},\dots,\beta_{l})=(\alpha_{1},\dots,\alpha_{{k-1}},\alpha_{k},\beta_{1},\beta_{2},\dots,\beta_{l}),$$
\item Fusion: for two non-empty words, we set 
\begin{align*}
(\alpha_{1},\dots,\alpha_{k}).(\beta_{1},\dots,\beta_{l})&=(\beta_{1},\dots,\alpha_{{k-1}},\alpha_{k}\otimes\beta_{1},\beta_{2},\dots,\beta_{l})\\
&=\left(\beta_{1},\dots,\alpha_{{k-1}},\bigoplus_{\gamma\subset\alpha_k\otimes\beta_1}\gamma,\beta_{2},\dots,\beta_{l}\right)\\
&=\bigoplus_{\gamma\subset\alpha_k\otimes\beta_1}\left(\beta_{1},\dots,\alpha_{{k-1}},\gamma,\beta_{2},\dots,\beta_{l}\right),
\end{align*}
where each $\gamma$ appears in the direct sum with its multiplicity $\gamma\subset\alpha_k\otimes\beta_1$.
\end{enumerate}
\end{defi}

\begin{thm}\label{nonalter}
The irreducible corepresentations of $\mathbb{G}\wr_*S_N^+$ can be labelled $\omega(x)$ with $x\in M$, with involution $\overline{\omega(x)}=\omega({\overline{x}})$ and the fusion rules:
$$\omega(x)\otimes\omega(y)=\sum_{x=u,t\ ;\ y=\overline{t},v}\omega({u,v})\ \oplus \displaystyle\sum_{\substack{x=u,t\ ;\ y=\overline{t},v\\ u\ne\emptyset,v\ne\emptyset}} \omega({u.v})$$ and we have for all $\alpha\in \Irr(\mathbb{G})$, $\omega(\alpha)=r(\alpha)\ominus \delta_{\alpha,1_{\mathbb{G}}}1$.
\end{thm} 

\begin{proof}
To prove this theorem, we shall describe the irreducible corepresentations and fusion rules in $\mathbb{H}$. This follows from Theorem \ref{frprodcrucial} used in the following lemma:


\begin{lem}\label{fusdansprod}
The irreducible representations of $\mathbb{H}$ can be labelled $\omega'(\alpha_1,\dots,\alpha_k)$, $(\alpha_1,\dots,\alpha_k)\in \langle \Irr(\mathbb{G})\rangle$, with involution $\overline{\omega'}(\alpha_1,\dots,\alpha_k)=\omega'(\bar\alpha_k,\dots,\bar\alpha_1)$ and fusion rules:
\begin{align*}
\omega'(\alpha_1,\dots,\alpha_k)\otimes\omega'(\beta_1,\dots,\beta_l)=\omega'(\alpha_1,\dots,\alpha_k,&\beta_1,\dots,\beta_l)\oplus\bigoplus_{\gamma\subset\alpha_k\otimes\beta_1}\omega'(\alpha_1,\dots,\gamma,\dots,\beta_l)\\
&\oplus\delta_{\bar\beta_1,\alpha_k}\omega'(\alpha_1,\dots,\alpha_{k-1})\otimes\omega'(\beta_2,\dots,\beta_l)
\end{align*}
where for all $\alpha\in\Irr(\mathbb{G})$, $\omega'(\alpha)=b\otimes\alpha\otimes b\ominus\delta_{\alpha,1_{\mathbb{G}}}1$.
\end{lem}

\begin{proof}

By Theorem \ref{frprodcrucial}, the irreducible representations of $\mathbb{G}*SU_q(2)$ can be indexed by $M'=\Irr(\mathbb{G})*\Irr(SU_q(2))$. The words $b^0$ and $1_{\mathbb{G}}$ are identified to the neutral element $\emptyset\in M'$. The elements of this free product can be written as ``reduced" words $w=b^{l_1}\gamma_1b^{l_2}\dots \gamma_{{k-1}}b^{l_k}$ with 
\begin{enumerate}
\item[$\bullet$] $l_1,l_k\ge0$, $\forall i\in\{2,\dots,k-1\}$ $l_i\ge1$, $\gamma_i\in \Irr(\mathbb{G})\setminus\{1_{\mathbb{G}}\}$ in the case $k>1$,
\item[$\bullet$] $w=b^l$ in the case $k=1$ for some $l\ge0$.
\end{enumerate}
The involution on $M'$ is given by $\overline{b^{l_1}\gamma_1b^{l_2}\dots \gamma_{l_{k-1}}b^{l_k}}=b^{l_k}\bar\gamma_{k-1}b^{l_{k-1}}\dots \bar\gamma_{{1}}b^{l_1}$. The definition of $\mathbb{H}$ implies that the irreducible representations of $\mathbb{H}$ are sub-representations of the tensor products
$$(b\otimes\alpha_1\otimes b)\otimes\dots\otimes(b\otimes\alpha_K\otimes b), K\in\N, \alpha_i\in\Irr(\mathbb{G})$$
which decomposes as a direct sum of irreducible representations $b^{l_1}\gamma_1b^{l_2}\dots \gamma_{l_{k-1}}b^{l_k}\in M'$ with 
\begin{enumerate}
\item[$\bullet$] $l_1,l_k$ odd integers and $l_i\ge2$ even integers for all $1<i<k$,
\item[$\bullet$] For all $i$, $\gamma_i\subset \alpha_{i_1}\otimes\dots\otimes\alpha_{i_t}$ for some $i_1,\dots,i_t\in\{1,\dots,k\}$, 
\item[$\bullet$] $w=b^{2K}$ obtained in the case $1\subset\alpha_1\otimes\dots\otimes\alpha_K$.
\end{enumerate}
We denote $N'\subset M'$ the set of all such words. Note that $\Irr(\mathbb{H})=N'$ is generated by the words $b\alpha b$, $\alpha\in \Irr(\mathbb{G})$. The description of the fusion rules binding irreducible representations in $N'$ follows from Theorem \ref{frprodcrucial}:
\begin{equation}\label{altto}
vb\alpha\otimes \beta bw=\sum_{1\ne t\subset\alpha\otimes\beta}vbtbw+\delta_{\alpha,\bar\beta}(v\otimes w),
\end{equation} 
for all $v,w\in N'$.
We then have a bijection $$\psi : \langle\Irr(\mathbb{G})\rangle\to N'=\Irr(\mathbb{H}), \omega'(\alpha_1,\dots,\alpha_k)\mapsto [b\alpha_1 b^2\dots b^2\alpha_kb]$$ where $[b\alpha_1 b^2\dots b^2\alpha_kb]$ denotes the ``reduced" word obtained by deleting the letters $\alpha_i=1_{\mathbb{G}}$. For any $p,q\in N$ and $\alpha,\beta\in\Irr(\mathbb{G})$, we have by (\ref{altto}):
\begin{align*}
pb\alpha b\otimes b\beta b q&=pb\alpha b^2\beta b q\oplus \sum_{1\ne t\subset\alpha\otimes\beta}pbt b q\oplus\delta_{\bar\beta,\alpha}p\otimes q\\
&=\psi(\omega'(p,\alpha,\beta,q))\oplus \sum_{1\ne t\subset\alpha\otimes\beta}\psi(\omega'(p,t,q))\oplus\delta_{\bar\beta,\alpha}\psi(\omega'(p))\otimes\psi(\omega'(q)).
\end{align*}
Hence, the fusion rules for $\mathbb{H}$ can indeed be described as in the statement. The statement on the involution follows from the following calculation:
\begin{align*}
\overline{\psi\left(\omega'(\alpha_1,\dots,\alpha_k)\right)}&=\overline{b^{l_1}\alpha_1b^{l_2}\dots \alpha_{l_{k-1}}b^{l_k}}\\
&=b^{l_k}\bar\alpha_{k-1}b^{l_{k-1}}\dots \bar\alpha_{1}b^{l_1}\\
&=\psi\left(\omega'(\alpha_k,\dots,\alpha_1)\right)\\
&=\psi\left(\overline{\omega'}(\alpha_1,\dots,\alpha_k)\right).
\end{align*}
\end{proof}
The description of the irreducible representations of $\mathbb{G}\wr_*S_N^+$ then follows from the monoidal equivalence $\mathbb{G}\wr_*S_N^+\simeq_{mon}\mathbb{H}$.
\end{proof}
The dimension of the irreducible representations of $\mathbb{G}\wr_*S_N^+$ are computed in the appendix of this paper. 

\section{Operator algebraic properties for free wreath product quantum groups}\label{secalgop}
In this section, we collect several corollaries that we can deduce from the monoidal equivalence we proved in the previous section. In particular, we study approximation properties and certain stability results under free wreath product of compact quantum groups. One can refer to \cite{Bra12}, \cite{Fre13}, \cite{DFS13} for definitions and several results in the cases of free compact quantum groups. We will use the following definition form \cite{CFY13}:

\begin{defi}
A discrete quantum group $\widehat{\mathbb{G}}$ is said to have the central almost completely positive approximation property (central ACPAP) if there is a net of central functionals $\phi_i$ on $\Pol(\mathbb{G})$ such that:
\begin{enumerate}
\item for all $i$, the convolution operator $T_{\phi_i}=(\phi_i\otimes \id)\circ\Delta$ induces a unital completely positive map on $C_r(\mathbb{G})$,
\item for all $i$, the operator $T_{\phi_i}$ is approximated by finitely supported central multipliers with respect to the $cb$-norm,
\item for any representation $\alpha\in \mathbb{G}$, $\lim_i\phi_i(\chi_{\alpha})/d_{\alpha}$=1.
\end{enumerate}
\end{defi}

It is proved in \cite{CFY13} that this property both implies the Haagerup property and the $W^*CCAP$ for $L^{\infty}(\mathbb{G})$ with respect to the Haar state. Notice that in our case, $\mathbb{G}$ is of Kac type, and the central ACPAP is equivalent to the the ACPAP without assuming the central property of the states. This follows from an averaging method from \cite{Bra12}.

Recall that we denote by $u=(u_{ij})_{ij}$ the generating magic unitary of $C(S_N^+)$ of dimension $N$. In this section, we use the results of \cite{CFY13}, including:

\begin{thm}(\/\cite[Theorem 22]{CFY13}\/)
The dual of $SU_{q}(2)$ has the central ACPAP for all $q\in[-1,1]\setminus\{0\}$.
\end{thm}

The following result concerns exactness.

\begin{thm}
Let $\mathbb{G}$ be a matrix compact quantum group of Kac type and $N\ge4$ and $0<q\le1$ such that $q+q^{-1}=\sqrt{N}$. Then the following are equivalent:
\begin{enumerate}
\item $C_r(\mathbb{G})$ is exact.
\item $C_r(\mathbb{G}\wr_*S_N^+)$ is exact.
\end{enumerate}
\end{thm}

\begin{proof}
Recall that exactness for the reduced $C^*$-algebras of compact quantum groups is stable under monoidal equivalence \cite[Theorem 6.1]{VV07}. Since $SU_{q}(2)$ is co-amenable, we have that $C_r(SU_{q}(2))=C_u(SU_{q}(2))$ is nuclear and thus exact. We then obtain the result by Theorem \ref{resmono}.
\end{proof}

\begin{thm}
Let $\mathbb{G}$ be a matrix compact quantum group and $N\ge4$. Then the following are equivalent:
\begin{enumerate}
\item\label{Gimp} The dual of $\mathbb{G}$ has the central ACPAP.
\item\label{Courimp} The dual of $\mathbb{G}\wr_*S_N^+$ has the central ACPAP.
\end{enumerate}
\end{thm}

\begin{proof}
We first prove that (\ref{Gimp})$\Rightarrow$(\ref{Courimp}). By Theorem \ref{resmono}, we have $\mathbb{G}\wr_*S_N^+\simeq_{mon}\mathbb{H}$ with $C(\mathbb{H})\subset C(\mathbb{G})*C(SU_{q}(2))$ for some $0<q\le1$ such that $q+q^{-1}=\sqrt{N}$. We know by \cite[Lemma 20, Proposition 21]{CFY13} that the central ACPAP is stable by taking free products and discrete quantum subgroups. The first implication then follows.

We now prove (\ref{Courimp})$\Rightarrow$(\ref{Gimp}). If the dual of $\mathbb{G}\wr_*S_N^+$ has the central ACPAP then there exists a sequence of central multipliers $T_{\phi_i} : \Pol(\mathbb{G}\wr_*S_N^+)\to \Pol(\mathbb{G}\wr_*S_N^+)$ $$T_{\phi_i}=(\phi_i\otimes \id)\circ \Delta=\sum_{w\in \Irr(\mathbb{G}\wr_*S_N^+)}\frac{\phi_i(\chi_w)}{d_w}p_w$$ such that
\begin{enumerate}
\item[$\bullet$] for all $i$, $T_{\phi_i}$ induces a unital, completely positive map on $C_r(\mathbb{G}\wr_*S_N^+)$,
\item[$\bullet$] for all $i$, $T_{\phi_i}$ is approximated by finitely supported central multipliers with respect to the $cb$-norm,
\item[$\bullet$] for any $w\in \Irr(\mathbb{G}\wr_*S_N^+)$, $\lim_i\phi_i(\chi_w)/d_w=1$.
\end{enumerate}
Consider $S_{i} : \Pol(\mathbb{G})\to \Pol(\mathbb{G}\wr_*S_N^+)$ $$S_i:=T_{\phi_i}\circ\bar\nu_1$$ where $\bar\nu_1 : C(\mathbb{G})\to C(\mathbb{G}\wr_*S_N^+)$ is the injective morphism of inclusion of (the first copy of) $C(\mathbb{G})$ in $C(\mathbb{G}\wr_*S_N^+)$, see Remark \ref{inclinj}. We simply denote $\bar\nu_1=\nu_1$. Let $a\in \Pol(\mathbb{G})$: it is a linear combination of coefficients $\alpha_{jk}$, with $\alpha\in \Irr(\mathbb{G})$. On such coefficients $\alpha_{jk}\ne1$, $S_i$ acts as follows:
\begin{align*}
S_i(\alpha_{jk})&=T_{\phi_i}(\nu_1(\alpha_{jk}))=(\phi_i\otimes \id)(\sum_{t,s}\nu_1(\alpha_{js})u_{1t}\otimes\nu_t(\alpha_{sk}))\\
&=\sum_{t,s}\phi_i(\nu_1(\alpha_{js})u_{1t})\nu_t(\alpha_{sk})\\
&=\frac{\phi_i\left(\nu_1\left(\sum_{j}\alpha_{jj}\right)\sum_ru_{rr}\right)}{d_{\alpha}N}\nu_1(\alpha_{jk}),
\end{align*}
where the last equality holds since $\phi_i$ is central. Then, $\nu_1(\Pol(\mathbb{G}))$ is stable under the action of $T_{\phi_i}$. We deduce that $S_i : C_r(\mathbb{G})\to C_r(\mathbb{G})$ is a sequence of unital completely positive maps (by composition of u.c.p. maps $T_{\phi_i}$, $\nu_1$), $$S_i=\sum_{\alpha\in \Irr(\mathbb{G})}\frac{\phi_i(\nu_1(\chi_{\alpha})\chi_u)}{d_{\alpha}N}p_{\alpha}$$ where $\chi_u$ is the character of the fundamental representation of $S_N^+$. Notice that $$\frac{\phi_i(\nu_1(\chi_{\alpha})\chi_u)}{d_{\alpha}N}\to_i1.$$

To conclude, recall that each $T_{\phi_i}$ can be approximated in $cb$-norm by central multipliers $t^i_j$ with finite supports on $\Irr(\mathbb{G}\wr_*S_N^+)$. Then $t_j^i$ is zero except on the coefficients of a finite set of words $(\alpha_1,\dots,\alpha_k)\in \Irr(\mathbb{G}\wr_*S_N^+)$ and then except on a finite number of letters $\alpha\in \Irr(\mathbb{G})$. The converse implication then follows from these observations.
\end{proof}

We end this paper by some concluding remarks and open questions:
\begin{rque}
In particular, the dual of $\mathbb{G}$ has the Haagerup property if and only if the dual of $\mathbb{G}\wr_*S_N^+$ has the Haagerup property.
\end{rque}

\begin{rque}
If the dual of $\mathbb{G}$ has the central ACPAP then $L^{\infty}(\mathbb{G}\wr_*S_N^+)$ has the $W^*CCAP$. Combined with exactness, one could try and prove that $L^{\infty}(\mathbb{G}\wr_*S_N^+)$ has the Akemann-Ostrand property to conclude that it has no Cartan subalgebras. We indeed already know that $L^{\infty}(\mathbb{G}\wr_*S_N^+)$ is non injective since $\mathbb{G}$ (and then $\mathbb{G}\wr_*S_N^+$) is of Kac type and non coamenable.
\end{rque}

\begin{rque}
One could try to find the fusion rules of $\mathbb{G}\wr_*S_N^+$ for non-Kac type compact quantum groups. Similar argument as in the case where $\mathbb{G}$ is the dual of a discrete group, \cite{Lem13b}, should apply to prove that in most cases $C_r(\mathbb{G}\wr_*S_N^+)$ is simple with unique trace and that $L^{\infty}(\mathbb{G})$ is a full type $II_1$-factor. In particular, fullness for $L^{\infty}(\mathbb{G}\wr_*S_N^+)$ would imply the non injective of this von Neumann algebra. Hence, the Akemann-Ostrand property could also be investigated in this setting in order to prove the absence of Cartan subalgebras for $L^{\infty}(\mathbb{G}\wr_*S_N^+)$.
\end{rque}

\section*{Appendix - Dimension formula}
In this section we obtain a dimension formula for the irreducible corepresentations of $\mathbb{G}\wr_*S_N^+$. This is an analogue formula as in the cases $\mathbb{G}=\widehat{\Gamma}$ see \cite[Theorem 9.3]{BV09}, \cite[Corollary 2.2]{Lem13}. In this subsection $\mathbb{G}$ is any compact matrix quantum group of Kac type, $N\ge4$.

By universality, there is a morphism $$\pi : C(\mathbb{G}\wr_*S_N^+)\to C(S_N^+),\ \omega_{ijkl}=v_{kl}^{(i)}u_{ij}\mapsto u_{ij}$$ induced by the morphisms on each factor of the free product $C(\mathbb{G})^{*N}*C(S_N^+)$, $$\pi^{(i)}=\left(\epsilon_{\mathbb{G}}^{(i)}\right)* \id$$ and which passes to the quotient $C(\mathbb{G})*_wC(S_N^+)$ since the image of $\epsilon_{\mathbb{G}}$ lies in $\mathbb{C}$. It corresponds a functor $\pi : \Irr(\mathbb{G}\wr_*S_N^+)\to \Irr(S_N^+)$, sending $r(\alpha), \alpha\in \Irr(\mathbb{G})$ to $u^{\oplus d_{\alpha}}$ where $d_{\alpha}$ denotes the dimension of the $\mathbb{G}$-representation $\alpha$

We denote by $\chi_{\rho}=(\id\otimes \text{Tr})(\rho)$ character associated to $\rho\in\Irr(\mathbb{G}\wr_*S_N^+)$. It is proved in \cite[Proposition 4.8]{Bra12} that the central algebra $C(S_N^+)_0=C^*-\langle\chi_k : k\in\N\rangle$ is isomorphic with $C([0,N])$ via $\chi_k\mapsto  A_{2k}(\sqrt{X})$ where $(A_k)_{k\in\N}$ is the family of dilated Tchebyshev polynomials defined inductively by $A_0=1, A_1=X$ and $A_1A_k=A_{k+1}+A_{k-1}$. 

We use an alternative description of the fusion rules in $\mathbb{G}\wr_*S_N^+$ which can be readily obtained from the proof of Proposition \ref{fusdansprod} where this description is obtained for the compact quantum group $\mathbb{H}$ monoidally equivalent to $\mathbb{G}\wr_*S_N^+$. We can write any $\rho\in\Irr(\mathbb{G}\wr_*S_N^+)\subset\Irr(\mathbb{G})*\Irr(SU_q(2))$ as follows $\rho=b^{l_1}\alpha_1b^{l_2}\dots b^{l_{k-1}}\alpha_{k-1}b^{l_k}$ with
\begin{enumerate}
\item[$\bullet$] $l_1,l_k$ odd integers and $l_i\ge2$ even integers for all $1<i<k$,
\item[$\bullet$] For all $i$, $\alpha_i\ne1_{\mathbb{G}}$,
\item[$\bullet$] $w=a^{2l}$ for some $l\ge0$ in the case $k=1$.
\end{enumerate}
and with the fusion rules recursively obtained by
\begin{equation*}
vb\alpha\otimes \beta bw=vb(\alpha\otimes\beta)bw+\delta_{\alpha,\bar\beta}(v\otimes w),
\end{equation*} 
for all words $v,w\in\Irr(\mathbb{G}\wr_*S_N^+)$. We also use the notation of Theorem \ref{nonalter}.


\begin{prp}
Let $\chi_{\rho}$ be the character of an irreducible corepresentation $\rho\in \Irr(\mathbb{G}\wr_*S_N^+)$. Write $\rho=b^{l_1}\alpha_1\dots b^{l_k}$. Then, identifying $C(S_N^+)_0$ with $C([0,N])$, the image of $\chi_{\rho}$ by $\pi$, say $P_{\rho}$, satisfies $$P_{\rho}(X)=\pi(\chi_{\rho})(X)=\prod_{i=1}^{k-1}d_{\alpha_i}\prod_{i=1}^kA_{l_i}(\sqrt{X}).$$
\end{prp}

\begin{proof}
We shall prove this proposition by induction on the even integer $\sum_{i=1}^kl_i$ using the description of the fusion rules above and a recursion formula satisfied by the Tchebyshev polynomials, see \cite{Lem13}.
 
Let HR($\lambda$) be the following statement: $\pi(\chi_{\rho})(X)=\prod_{i=1}^kA_{l_i}(\sqrt{X})$ for any $\rho=b^{l_1}\alpha_1\dots \alpha_{k-1}b^{l_k}$ such that $2\le \sum_{i}l_i\le \lambda$.

To initialize the induction let us consider the irreducible corepresentations $b\alpha b\equiv r(\alpha)$, $\alpha\in \Irr(\mathbb{G})\setminus\{1_{\mathbb{G}}\}$. It is sent via $\pi$ onto $u^{\oplus d_{\alpha}}=d_{\alpha}1\oplus u^{(1)\oplus d_{\alpha}}$. Thus, in term of characters, we have $$\pi(\chi_{b\alpha b})(X)=d_{\alpha}(1+(X-1))=d_{\alpha}X=d_{\alpha}A_1(\sqrt{X})A_1(\sqrt{X}).$$

Consider now $b^2\equiv\omega(1_{\mathbb{G}})= r(1_{\mathbb{G}})\ominus 1$. It is sent by $\pi$ onto $v^{(1)}$. Thus $P_{b^2}(X)=\pi(\chi_{b^2})(X)=X-1=A_2(\sqrt{X})$. $HR(2)$ is then proved.


Now assume HR($\lambda$) holds: $$\pi(\chi_{\rho})(X)=\prod_{i=1}^{k-1} d_{\beta_i}\prod_{i=1}^kA_{l_i}(\sqrt{X})$$ for any $\rho=b^{l_1}\beta_1\dots \beta_{k-1}a^{l_k}$ such that $2\le \sum_{i}l_i\le \lambda$. We now show HR($\lambda+2$).

Let $\rho=b^{L_1}\alpha_1\dots b^{L_K}$, with $\sum_iL_i=\lambda+2$. In order to use HR($\lambda$), we must ``break'' $\rho$ using the fusion rules as in the examples above. 
Then, essentially, one has to distinguish the cases $L_K=1, L_K=3$ and $L_K\ge5$ (in the case $L_K\ge5$ we can ``break $\rho$ at $b^{L_{K}}$" but in the other cases we must use $b^{L_{K-1}}$ or $b^{L_{K-2}}$ if they exist, that is if there are enough factors $b^{L_i}$). So first, we deal with two special cases below, in order to have ``enough'' factors $b^L$ in $\rho$ in the sequel.

\begin{enumerate}

\item[-] If $K=1$ i.e. $L_K=\lambda+2$, write: $$\rho=b^{\lambda+2}=(b^{\lambda}\otimes b^2)\ominus (b^{\lambda-1}\otimes b)=(b^{\lambda}\otimes b^2)\ominus b^{\lambda}\ominus b^{\lambda-2}.$$ Then using the hypothesis of induction and \cite[Proposition 1.7]{Lem13}, we get 

\begin{align*}
\pi(\chi_{\rho})(X)&=A_{\lambda}A_2(\sqrt{X})-A_{\lambda}(\sqrt{X})-A_{\lambda-2}(\sqrt{X})\\
&=A_{\lambda}A_2(\sqrt{X})-(A_{\lambda}(\sqrt{X})+A_{\lambda-2}(\sqrt{X}))\\
&=A_{\lambda}A_2(\sqrt{X})-A_{\lambda-1}A_1(\sqrt{X})\\
&=A_{\lambda+2}(\sqrt{X}).
\end{align*}
(Notice that if $\lambda=2$ one has $\lambda-2=0$ and $b^4=(b^2\otimes b^2)\ominus (b\otimes b)=(b^2\otimes b^2)\ominus b^2\ominus 1$ so that the result we want to prove then is still true.)

\item[-] If $K=2, \alpha_1\not\simeq1_{\mathbb{G}}$, write $\rho=b^{L_1}\alpha b^{L_2}$. We have $L_1+L_2= \lambda+2\ge4$ and $L_1,L_2$ are odd hence $L_1$ or $L_2\ge3$, say $L_1\ge3$. Write $$b^{L_1}\alpha b^{L_2}=(b^2\otimes b^{L_1-2}\alpha b^{L_2})\ominus (b\otimes b^{L_1-3}\alpha b^{L_2}).$$

If $L_1=3$ then the tensor product $b\otimes b^{L_1-3}z^Jb^{L_2}$ is equal to $b\alpha b^{L_2}$ hence $\rho=b^3\alpha b^{L_2}$ satisfies 
\begin{align*}
\pi(\chi_{\rho})(X)&=d_{\alpha}A_2A_{1}A_{L_2}(\sqrt{X})-d_{\alpha}A_1A_{L_2}(\sqrt{X})\\
&=d_{\alpha}A_3(\sqrt{X})A_{L_2}(\sqrt{X}).
\end{align*}

If $L_1>3$ (i.e. $L_1\ge5$), then the tensor product $b\otimes b^{L_1-3}\alpha b^{L_2}$ is equal to $b^{L_1-2}\alpha b^{L_2}\oplus b^{L_1-4}\alpha b^{L_2}$. We get
\begin{align*}
\pi(\chi_{\rho})(X)&=d_{\alpha}A_2A_{L_1-2}A_{L_2}(\sqrt{X})-d_{\alpha}A_{L_1-2}A_{L_2}(\sqrt{X})-d_{\alpha}A_{L_1-4}A_{L_2}(\sqrt{X})\\
&=d_{\alpha}A_{L_1}(\sqrt{X})A_{L_2}(\sqrt{X}).
\end{align*}

\item[-] From now on, we suppose that there are more than three factors $b^{L_i}$ in $\rho$ i.e. $K\ge3$. We will have to distinguish three cases: $L_K=1, L_K=3$ and $L_K\ge5$.

If $5\le L_K<\sum_iL_i$, write $L_K=m_K+2$. Then we have $m_K\ge3$, so 
\begin{align*}
b^{L_1}\alpha_1\dots b^{L_K}&=b^{L_1}\alpha_{1}\dots b^{m_K+2}\\
&=(b^{L_1}\alpha_1\dots b^{m_K}\otimes b^2)\ominus (b^{L_1}\alpha_1\dots b^{m_K-1}\otimes b)\\
&=(b^{L_1}\alpha_1\dots b^{m_K}\otimes b^2)\ominus b^{L_1}\alpha_1\dots b^{m_K}\ominus b^{L_1}\alpha_1\dots b^{m_K-2}.
\end{align*}
Then 
\begin{align*}
\pi(\chi_{\rho})(X^2)&=\prod_{i=1}^{k-1}d_{\alpha_i}\left(A_{L_1}\dots A_{L_{K-1}}A_{m_k}A_2(X)-A_{L_1}\dots A_{m_K}(X)-A_{L_1}\dots A_{m_K-2}(X)\right)\\
&=\prod_{i=1}^{k-1}d_{\alpha_i}A_{L_1}\dots A_{L_{K-1}}A_{L_K}(X).\\
\end{align*}

If $m_K=1$, i.e. $L_K=3$, we proceed in the same way using
\begin{align*}
b^{L_1}\alpha_1\dots \alpha_{K-1}b^3=(b^{L_1}\alpha_{1}\dots b\otimes b^2) \ominus b^{L_1}\alpha_{1}\dots \alpha_{K-1}b.
\end{align*}

To conclude the induction, one has to deal with the case $L_K=1$. We have to distinguish the following cases:

If $L_{K-1}\ge4$. We have
\begin{align*}
b^{L_1}\alpha_{1}\dots b^{L_{K-1}}\alpha_{K-1}b&=(b^{L_1}\alpha_{1}\dots b^{L_{K-1}-1}\otimes b\alpha_{K-1}b) \ominus (b^{L_1}\alpha_{1}\dots b^{L_{K-1}-2}\otimes \alpha_{K-1}b)\\
&=(b^{L_1}\alpha_{1}\dots b^{L_{K-1}-1}\otimes b\alpha_{K-1}b) \ominus b^{L_1}\alpha_{1}\dots b^{L_{K-1}-2}\alpha_{K-1}b.
\end{align*}
Then
\begin{align*}
\pi(\chi_{\rho})(X)&=\prod_{i=1}^{k-1}d_{\alpha_i}\left(A_{L_1}\dots A_{L_{K-1}-1}A_{1}A_1(\sqrt{X})-A_{L_1}\dots A_{L_{K-1}-2}A_1(\sqrt{X})\right)\\
&=\prod_{i=1}^{k-1}d_{\alpha_i}A_{L_1}\dots A_{L_{K-1}}A_1(\sqrt{X}).\\
\end{align*}

If $L_{K-1}=2$ and $\alpha_{K-1}\simeq\bar\alpha_{K-2}$, we use 

\begin{align*}
b^{L_1}\alpha_{1}\dots &b^{L_{K-2}}\alpha_{K-2}b^2\alpha_{K-1}b\\
&=(b^{L_1}\alpha_{1}\dots b^{L_{K-2}}\alpha_{K-2}b\otimes b\alpha_{K-1}b)\ominus b^{L_1}\alpha_{1}\dots b^{L_{K-2}}(\alpha_{K-2}\otimes\alpha_{K-1})b\ominus b^{L_1}z^{J_1}\dots b^{L_{K-2}-1}.\\
\end{align*}
We write:
\begin{align*}
\alpha_{K-2}\otimes\alpha_{K-1}=1\oplus\bigoplus_{\gamma\not\simeq1_{\mathbb{G}}}\gamma,
\end{align*}
and we get
\begin{align*}
&\pi(\chi_{\rho})(X)=\prod_{i=1}^{K-3}d_{\alpha_i}A_{L_i}(\sqrt{X})[d_{\alpha_{K-2}}d_{\alpha_{K-1}}A_{L_{K-2}}A_1^3(\sqrt{X})-\\
&\ \ \ \ \ -\sum_{\gamma\not\simeq1}d_{\gamma}A_{L_{K-2}}A_1(\sqrt{X})-A_{L_{K-2}+1}-A_{L_{K-2}-1}(\sqrt{X})]\\
&=\prod_{i=1}^{K-3}d_{\alpha_i}A_{L_i}(\sqrt{X})[d_{\alpha_{K-2}}d_{\alpha_{K-1}}A_{L_{K-2}}A_1^3(\sqrt{X})-\\
&\ \ \ \ \ -\sum_{\gamma\not\simeq1}d_{\gamma}A_{L_{K-2}}A_1(\sqrt{X})-A_1A_{L_{K-2}}(\sqrt{X})]\\
&=\prod_{i=1}^{K-3}d_{\alpha_i}A_{L_i}(\sqrt{X})[d_{\alpha_{K-2}}d_{\alpha_{K-1}}A_{L_{K-2}}A_1^3(\sqrt{X})-\\
&\ \ \ \ \ -d_{\alpha_{K-2}}d_{\alpha_{K-1}}A_{L_{K-2}}A_1(\sqrt{X})]\\
&=\prod_{i=1}^{K-3}d_{\alpha_i}A_{L_i}(\sqrt{X})d_{\alpha_{K-2}}d_{\alpha_{K-1}}A_{L_{K-2}}A_2A_1(\sqrt{X}).
\end{align*}

The last case to deal with is $L_{K-1}=2$ and $\alpha_{K-1}\not\simeq\bar\alpha_{K-2}$, and again we can conclude thanks to
\begin{align*}
b^{L_1}\alpha_1\dots \alpha_{K-2}b^2\alpha_{K-1}b&=(b^{L_1}\dots b^{L_{K-2}}\alpha_{K-1}b\otimes b\alpha_{K-1}b) \ominus b^{L_1}\alpha_1\dots b^{L_{K-2}}(\alpha_{K-2}\otimes\alpha_{K-1})b.\\
\end{align*}
\end{enumerate}
\end{proof}

\begin{crl}(\/\cite[Corollary 2.2]{Lem13}\/)
Let $\rho$ be an irreducible corepresentation of $\mathbb{G}\wr_*S_N^+$ with $\rho=b^{l_1}\alpha_1\dots b^{l_k}$. Then $$\emph{dim}(\rho)=\prod_{i=1}^{k-1}d_{\alpha_i}\prod_{i=1}^kA_{l_i}(\sqrt{N}).$$
\end{crl}

\section*{Acknowledgements}
We are grateful to Pierre Fima for suggesting to us the monoidal equivalence argument, fundamental in this article. We want to thank Roland Vergnioux for the time he spent discussing the arguments of this article. The second author wishes to thank Philippe Biane for his suggestions on the present article. We wish to thank Roland Speicher and his team as well as Uwe Franz and Campus France (Egide) who made the joint work of the authors possible.

\bibliographystyle{alpha}
\nocite{*}
\bibliography{Wreathb}

\end{document}